\documentclass[a4paper]{amsart}

\usepackage{graphics,amssymb,enumerate}
\usepackage{hyperref}
\usepackage{color}

\newcommand {\bd} {\begin{displaymath}}
\newcommand {\ed} {\end{displaymath}}
\newcommand {\be} {\begin{equation}}
\newcommand {\ee} {\end{equation}}
\makeatletter
\newcommand*{\Rom}[1]{\expandafter\@slowromancap\romannumeral #1@}

\makeatother
\newcommand {\bea} {\begin{eqnarray}}
\newcommand {\eea} {\end{eqnarray}}

\newcommand{\ts}{\textsuperscript}
\newcommand{\ga}{\alpha}

\newtheorem{theorem}{Theorem}
\newtheorem{proposition}{Proposition}

\newtheorem{example}{Example}
\newtheorem{remark}{Remark}
\newtheorem{definition}{Definition}

\def\g{\mathfrak{g}}

\begin{document}
\title{Characteristic and Coxeter polynomials for affine Lie algebras}

\author{Pantelis A.~Damianou}
\email{damianou@ucy.ac.cy}

\author{Charalampos A. Evripidou}
\email{cevrip02@ucy.ac.cy}
\thanks{
This work  was co-funded by the European Regional Development Fund
and the Republic of Cyprus through the Research Promotion Foundation (Project: PENEK/0311/30).
}

\address{Department of Mathematics and Statistics University of Cyprus
P.O.~Box 20537, 1678 Nicosia Cyprus}
\date{}

\begin{abstract}
	We compute the characteristic polynomials
	of affine Cartan, adjacency matrices and Coxeter polynomials
	of the associated Coxeter system using Chebyshev polynomials.
	We give explicit factorization of these polynomials as products
	of cyclotomic polynomials. Finally, we present several different methods
	of obtaining the exponents and Coxeter number for affine Lie algebras.
	In particular we compute the exponents and Coxeter number
	for each conjugacy class in the case of $A_n^{(1)}$.
\end{abstract}

\maketitle

\section{Introduction}

Complex finite dimensional semi-simple Lie algebras are classified via their root system.
All data of a root system is encoded in a matrix, the Cartan
matrix, or in a graph, the Dynkin diagram. This graph is constructed
along the same lines as the construction of the Coxeter graph from the Coxeter matrix.

A Cartan matrix is an $n \times n$-integer matrix $C$ which obeys
\begin{enumerate}
	\item[-] $C_{i,i}=2,$
	\item[-] $C_{i,j}=0 \Rightarrow C_{j,i}=0, \ \forall \, i,j$
	\item[-] $C_{i,j}\leq 0, \ \forall \, i\neq j$
	\item[-] $\det{C} >0.$
\end{enumerate}
Relaxing the last condition on Cartan
matrices we get the so called generalized Cartan matrices.
Generalized Cartan matrices are classified into three disjoint categories,
finite, affine and indefinite (see \cite{Kac} Chapter 4).
Finite are the usual Cartan matrices associated with complex semi-simple
finite dimensional Lie algebras, while
affine and indefinite, give rise to infinite dimensional Lie algebras.

For each generalized Cartan matrix the associated semi-simple complex Lie algebra $ \g $ is
algebraically generated by  $3n$  generators $\{x_i^{\pm},h_i:1\leq i \leq n\}$ which are
subjected to the relations

\begin{enumerate}
	\item $[h_i,h_j]=0 ,$
	\item $[h_i,x_j^{\pm}]=\pm C_{i,j}x_j^{\pm} , $
	\item $[x_i^{+},x_j^{-}]=\delta_{i,j}h_i ,$
	\item $(ad_{x_i^{\pm}})^{1-C_{j,i}}(x_j^{\pm}))=0 \ .$
\end{enumerate}
This set of relations is known as the  Chevalley-Serre relations.
An affine Cartan matrix is one for which  $\det{C}=0$ and each
proper principal minor of $C$ is positive. Thus each $(n-1)\times (n-1)$
submatrix of $C$ obtained by removing an $i\ts{th}$ row and $i\ts{th}$ column is a
Cartan matrix. Chevalley-Serre relations on affine Cartan matrices give rise to
affine Lie algebras. The main focus of this paper is on
affine Cartan matrices. This important subclass of generalized Cartan
matrices is characterized by the property that they are symmetrizable
and the corresponding symmetric matrices $DC$ are  positive semidefinite.
From now on $C$ will be an indecomposable finite or affine Cartan matrix.

With each Cartan or affine Cartan matrix $C$
we associate a graph, called the Dynkin diagram.
Its vertices $\{r_1,\ldots ,r_n \}$ correspond to the columns of $C$.
There are $C_{i,j}C_{j,i}$ edges between the vertices $r_i,r_j$ for $i\neq j$.
In case $C_{i,j} < C_{j,i}$ we put an arrow in the edge $(r_i,r_j)$
pointing to the vertex $r_j$. For an affine Cartan matrix it is customary to enumerate the
vertices as $\{r_0,r_1,\ldots,r_n\}$ so that the corresponding
Dynkin diagram has $n+1$ vertices.

\centerline{\bf Dynkin diagrams for affine Lie algebras}

$
	\put(25,20){\line(1,0){50}}
	\put(85,20){\line(1,0){50}}
	\put(205,20){\line(1,0){50}}
	\put(24.4,22){\line(4,1){115}}
	\put(255.6,22){\line(-4,1){115}}
	\put(20,20){\circle{10}}
	\put(80,20){\circle{10}}
	\put(140,20){\circle{10}}
	\put(200,20){\circle{10}}
	\put(260,20){\circle{10}}
	\put(140,50){\circle*{10}}
	\put(160,20){$\ldots$}
	\put(0,50){$\boldsymbol{A_{n}^{(1)}}$}
$

\vspace{-1ex}
$
	\put(25,42){\line(5,-2){50}}
	\put(25,-2){\line(5,2){50}}
	\put(85,20){\line(1,0){50}}
	\put(205,19){\line(1,0){50}}
	\put(205,21){\line(1,0){50}}
	\put(20.2,43.3){\circle{10}}
	\put(20.2,-3.3){\circle*{10}}
	\put(80,20){\circle{10}}
	\put(140,20){\circle{10}}
	\put(200,20){\circle{10}}
	\put(260,20){\circle{10}}
	\put(160,20){$\ldots$}
	\put(0,50){$\boldsymbol{B_{n}^{(1)}}$}
	\put(227,17.5){$>$}
$

\vspace{3ex}
$
	\put(25,19){\line(1,0){50}}
	\put(25,21){\line(1,0){50}}
	\put(85,20){\line(1,0){50}}
	\put(205,19){\line(1,0){50}}
	\put(205,21){\line(1,0){50}}
	\put(20,20){\circle*{10}}
	\put(80,20){\circle{10}}
	\put(140,20){\circle{10}}
	\put(200,20){\circle{10}}
	\put(260,20){\circle{10}}
	\put(160,20){$\ldots$}
	\put(0,30){$\boldsymbol{C_{n}^{(1)}}$}
	\put(227,17.5){$<$}
	\put(47,17.5){$>$}
$

\vspace{-1ex}
$
	\put(25,42){\line(5,-2){50}}
	\put(25,-2){\line(5,2){50}}
	\put(85,20){\line(1,0){50}}
	\put(205,20){\line(1,0){50}}
	\put(263.5,23.5){\line(5,2){50}}
	\put(263.5,16.5){\line(5,-2){50}}
	\put(20.2,43.3){\circle{10}}
	\put(20.2,-3.3){\circle*{10}}
	\put(80,20){\circle{10}}
	\put(140,20){\circle{10}}
	\put(200,20){\circle{10}}
	\put(260,20){\circle{10}}
	\put(319,45){\circle{10}}
	\put(319,-5){\circle{10}}
	\put(160,20){$\ldots$}
	\put(0,50){$\boldsymbol{D_{n}^{(1)}}$}
$

\vspace{3ex}
$
	\put(25,20){\line(1,0){50}}
	\put(85,20){\line(1,0){50}}
	\put(205,20){\line(1,0){50}}
	\put(145,20){\line(1,0){50}}
	\put(140,15){\line(0,-1){25}}
	\put(140,-20){\line(0,-1){25}}
	\put(20,20){\circle{10}}
	\put(80,20){\circle{10}}
	\put(140,20){\circle{10}}
	\put(200,20){\circle{10}}
	\put(260,20){\circle{10}}
	\put(140,-15){\circle{10}}
	\put(140,-50){\circle*{10}}
	\put(0,30){$\boldsymbol{E_{6}^{(1)}}$}
$

\vspace{2ex}
$
	\put(25,20){\line(1,0){30}}
	\put(65,20){\line(1,0){30}}
	\put(105,20){\line(1,0){30}}
	\put(145,20){\line(1,0){30}}
	\put(185,20){\line(1,0){30}}
	\put(225,20){\line(1,0){30}}
	\put(140,15){\line(0,-1){25}}
	\put(20,20){\circle*{10}}
	\put(60,20){\circle{10}}
	\put(100,20){\circle{10}}
	\put(140,20){\circle{10}}
	\put(180,20){\circle{10}}
	\put(220,20){\circle{10}}
	\put(260,20){\circle{10}}
	\put(140,-15){\circle{10}}
	\put(0,30){$\boldsymbol{E_{7}^{(1)}}$}
$

\vspace{3ex}
$
	\put(25,20){\line(1,0){25}}
	\put(60,20){\line(1,0){25}}
	\put(95,20){\line(1,0){25}}
	\put(130,20){\line(1,0){25}}
	\put(165,20){\line(1,0){25}}
	\put(200,20){\line(1,0){25}}
	\put(195,15){\line(0,-1){20}}
	\put(235,20){\line(1,0){25}}
	\put(20,20){\circle*{10}}
	\put(55,20){\circle{10}}
	\put(90,20){\circle{10}}
	\put(125.5,20){\circle{10}}
	\put(160,20){\circle{10}}
	\put(195.5,20){\circle{10}}
	\put(195.5,-10){\circle{10}}
	\put(230,20){\circle{10}}
	\put(265.5,20){\circle{10}}
	\put(0,30){$\boldsymbol{E_{8}^{(1)}}$}
$

\vspace{2ex}
$
	\put(25,0){\line(1,0){50}}
	\put(85,0){\line(1,0){50}}
	\put(205,0){\line(1,0){50}}
	\put(145,1){\line(1,0){51}}
	\put(145,-1){\line(1,0){51}}
	\put(20,0){\circle*{10}}
	\put(80,0){\circle{10}}
	\put(140,0){\circle{10}}
	\put(200,0){\circle{10}}
	\put(260,0){\circle{10}}
	\put(0,10){$\boldsymbol{F_{4}^{(1)}}$}
	\put(164,-2.6){$>$}
$

\vspace{2ex}
$
	\put(85,0){\line(1,0){50}}
	\put(145,0){\line(1,0){50}}
	\put(145,1.5){\line(1,0){50.5}}
	\put(145,-1.5){\line(1,0){50.5}}
	\put(140,0){\circle{10}}
	\put(200,0){\circle{10}}
	\put(80,0){\circle*{10}}
	\put(0,15){$\boldsymbol{G_{2}^{(1)}}$}
	\put(164,-2.6){$>$}
$

\newpage

For a Dynkin diagram $\Gamma$, in addition to the Cartan matrix $C$,
we  associate the \textit{ Coxeter adjacency
matrix}
which is  the matrix $A=2I-C$.
The \textit{characteristic polynomial} of $\Gamma$ is that of $A$
and the \textit{spectral radius} of $\Gamma$ is
\bd
	\rho \left(\Gamma\right)= \max \left\{\left|\lambda\right|:
	\lambda  \text{ is an eigenvalue of } A \right\}.
\ed

In this paper we use the following notation.
The subscript $n$  in all cases is equal to the degree of the
polynomial except that $Q_n(x)$ is of degree $2n$.

\begin{itemize}
	\item[-] $p_n(x)$ will denote the characteristic polynomial of the Cartan matrix,
	\item[-] $q_n(x)=\det{(2xI+A)}$,
	\item[-] $a_n(x)=q_n\left(\frac{x}{2}\right)$ will denote the characteristic polynomial of $-A$ and finally,
	\item[-] $Q_n(x)=x^n a_n \left( x + \frac{1}{x} \right). $
\end{itemize}
Note the relation between the polynomials $a_n, \ q_n$ and $p_n$:
\bd
	p_n(x)=a_n(x-2)=q_n\left(\frac{x}{2}-1\right).
\ed

We prove the following result:

\begin{theorem}
Let $C$ be the $n \times n$  affine Cartan matrix of an affine Lie algebra of type $X$.
Then $q_{n}$ is a polynomial related to Chebyshev polynomials as follows
\begin{gather*}
\text{for } X=A_{n-1}^{(1)}, \ \ q_{n}(x)= 2  \left( T_{n}(x) + (-1)^{n-1}   \right),\\
\text{for } X=B_{n-1}^{(1)}, \ \ q_{n}(x)= 2  \left( T_{n}(x)-T_{n-4}(x)   \right),\\
\text{for } X=C_{n-1}^{(1)}, \ \ q_{n}(x)= 2  \left( T_{n}(x)-T_{n-2}(x)   \right)\text{ and }\\
\text{for } X=D_{n-1}^{(1)}, \ \ q_{n}(x)=8x^2\left( T_{n-2}(x)-T_{n-4}(x)   \right),\\
\end{gather*}
where $T_n(x)$ is the $n$\ts{th} Chebyshev polynomial of first kind.
\end{theorem}

%

Using the fact that for bipartite Dynkin diagrams the spectrum  of $A$
is  the same as the spectrum of $-A$ it follows easily that the eigenvalues of the
Cartan matrix occur in pairs $\lambda $ and $4-\lambda$
(see e.g. \cite{Moody,Coxeter2,Dragos}).
In our case this happens in all cases except for $A_{n}^{(1)} ,\ n$ even.
In the bipartite cases, $a_n(x)$ is the characteristic polynomial of the Coxeter adjacency matrix.

Let $\g$ be a complex finite dimensional simple Lie algebra with Cartan matrix $C$ of
rank $n$ and simple roots $\Pi=\{\ga_1,\ldots,\ga_n\}$.
The Killing form on $\g$ induces an inner product on the real vector space $V$ with basis $\Pi$.
The Weyl group $W$ of $\g$ is a subgroup of $\operatorname{Aut}{V}$ which is generated by
reflections on $V$. Namely, for each root $\ga_i$ consider the reflection $\sigma_i$
through the hyperplane perpendicular to $\ga_i$
$$
\sigma_i : V \longrightarrow V , \ \ \ga \mapsto \ga-2\dfrac{(\ga,\ga_i)}{(\ga_i,\ga_i)}\ga_i.
$$
Then the Weyl group of $\g$ is $W=\langle \sigma_1,\sigma_2,\ldots,\sigma_n\rangle$.
The Cartan matrix $C$ satisfies $C_{i,j}=2\dfrac{(\ga_i,\ga_j)}{(\ga_j,\ga_j)}$
and therefore $ \sigma_i (\ga_j)=\ga_j-C_{j,i}\ga_i$.
The Weyl group of an affine Lie algebra with Cartan matrix $C$ of rank $n-1$ is
$W=\langle\sigma_1,\sigma_2,\ldots,\sigma_n\rangle$ where
$$
\sigma_i (\ga_j)=\ga_j-C_{j,i}\ga_i,
$$
is a ``reflection'' in the real vector space with basis $\{\ga_1,\ldots,\ga_n\}$.
If $z=(z_1,\ldots,z_n)$ is a left zero eigenvector of $C$
($z$ can be taken to be in  $ \mathbb{Z}^{n}$, see \cite{Kac}) and
$\ga=\sum_{k=1}^n z_k\ga_k$ then $\sigma_i(\ga)=\ga, \ \forall i=1,\ldots,n$.
Thus the Weyl group $W$ acts on $\{k\ga:k \in \mathbb{Z}\}$ as the identity.

A Coxeter polynomial $f_n$ is the characteristic polynomial of
$\sigma_{\pi(1)}\sigma_{\pi(2)}\ldots\sigma_{\pi(n)}\in gl(V)$
for some $\pi\in S_n$. When the Dynkin diagram does not contain cycles
the Coxeter polynomial is uniquely defined and for bipartite Dynkin
diagrams is closely related to the polynomial $Q_n(x)$;
the polynomial $Q_n(x)$ turned out to be $Q_n(x)=f_n(x^2)$
(see \cite{Moody}). For the case of $A_{n-1}^{(1)}$
there are $\left\lfloor \frac{n}{2} \right\rfloor$ different Coxeter polynomials.
For $n$ even, $Q_n\left(\sqrt{x}\right)$ is one of the Coxeter polynomials,
the one corresponding to the largest conjugacy class of the Coxeter transformations.
According to \cite{Coleman} the largest conjugacy class contains the Coxeter
transformations with the property that the set
$$
\left\{i:\pi^{-1}(i)>\pi^{-1}(i+1\pmod n), i=1,2,\ldots,n\right\}
$$
has the largest cardinality, i.e. contains $\frac{n}{2}$ elements.
For example the Coxeter transformation
$\sigma_1\sigma_3\ldots\sigma_{n-1}\sigma_2\sigma_4\ldots\sigma_n,$
which is the one considered in \cite{Moody}.

The roots of $Q_n$  are in the unit disk and therefore by a theorem of Kronecker
$Q_n(x)$ is a product of cyclotomic polynomials (see \cite{Damianou}).
We determine the factorization of $Q_n$ as a product of cyclotomic polynomials.
This factorization in turn determines  the factorization of $f_n$. The
irreducible factors of $Q_n$ are in one-to-one correspondence
with the irreducible factors of $a_n(x)$.

The roots of a Coxeter polynomial $f_n$, of affine type, are of the form
$e^{\frac{2m_j \pi i }{h}}$ where $ 0 \leq m_j \leq h$. The numbers $m_j$
are what we call \textit{affine exponents} and $h$ the \textit{affine Coxeter number}
associated with the Coxeter transformation $\sigma$.
These numbers are normally defined only for the bipartite case.
For $A_n^{(1)} , \ n$ odd one  defines them with
respect to the Coxeter polynomial corresponding to the largest conjugacy class
of the Coxeter elements. In this paper we also examine in detail the case of $A_n^{(1)}$
for $n$ both even and odd and we calculate the affine exponents and
affine Coxeter number for each conjugacy class.

These numbers are related to the Cartan matrix and give a universal formula
for the spectrum of the Dynkin diagram $\Gamma$
and the eigenvalues of the Cartan matrix.
For the bipartite case the spectrum is
$$
\left\{2 \cos{\frac{m_j \pi}{h}}: j=1,\ldots n \right\}
$$
and the eigenvalues of the Cartan matrix are
$\left\{4 \cos^2{\frac{m_j\pi}{2h}}: j=1,\ldots n \right\}$.

The affine exponents, affine Coxeter number of $X_{n}^{(1)}$ and the
roots of the corresponding simple Lie algebra $X_{n}$ are related in
an inquisitive way (see \cite{Moody}).
Let $\Pi=\{\ga_1,\ga_2,\ldots,\ga_{n}\}$ be the simple roots of
$X_n$, $V=\mathbb{R}\operatorname{-span}(\ga_1,\ga_2,\ldots,\ga_{n})$ and $\beta$
the branch root of $X_n$. Let $w_{\beta^\vee} \in V^*$ be the weight  corresponding to the
co root $\beta^\vee$. Then for some $c \in \mathbb N$ and a proper enumeration of $m_j$ we have
$$
c\cdot w_{\beta^\vee}=\displaystyle{\sum_{j=1}^{n}m_j\ga_j^\vee},
$$
where $c$ is the smallest integer such that $c\cdot w_{\beta^\vee}$
belongs to the co root lattice. The coefficient of $\beta^\vee$ is the affine
Coxeter number. Here we have identified $V$ with $V^*$  using  the inner product
induced by the Killing form. In this paper we show that this relation is valid
for all conjugacy classes in $A_n^{(1)}$. This case has not been considered previously. 

Steinberg, in his explanation of the MacKay correspondence,
showed in \cite{Steinberg} a mysterious relation between affine Coxeter polynomials
(for the simply laced Dynkin diagrams and later Stekolshchik in \cite{Steko}
for the multiple laced) and Coxeter polynomials of type $A_n$.
Each affine Coxeter polynomial is a product of Coxeter polynomials
of type $A_n$. From $X_n$ remove the branch root.
If $g(x)$ is the Coxeter polynomial of the reduced system then the Coxeter polynomial, $f(x)$
of $X^{(1)}_n$ is $f(x)=(x-1)^2g(x)$. The  affine exponents and affine Coxeter number
of an affine Lie algebra are easily computed using Steinberg's theorem.
In  table 1  we list  the affine exponents and affine
Coxeter number for affine Lie algebras.  Furthermore, we demonstrate
that the  method of Steinberg works also in the case of $A_n^{(1)}$ for $n$ even or odd a case which is not discussed in earlier literature.  

This paper is structured as follows.
In section 2 we give the definitions of Chebyshev and Cyclotomic polynomials and
some of their  properties. The polynomials $\Psi_n$ are also defined.
These are the irreducible factors of the polynomials
$a_n$ and are used to factor the polynomials $p_n$ and $q_n$.
In section 3 we compute and explicitly determine the
characteristic polynomials for  each affine Lie algebra.
In section 4 we define the affine Coxeter number
and affine exponents and we proceed in section 5 to determine them for each affine Lie algebra.
The exceptional affine Lie algebras are treated in section 6.
Section 7 deals with the method of Berman,  Lee and  Moody and the method
of Steinberg with special emphasis  to the case of $A_n^{(1)}$.

\section{Chebyshev and Cyclotomic polynomials}

\subsection{Chebyshev Polynomials}
A fancy way  to define Chebyshev polynomials of first ($T_n$) and second kind ($U_n$) is

\[
T_n\left(x\right)
=
\frac{1}{2}\cdot
\det{
\begin{pmatrix}
	2x&1&0&\cdots&0&0&0\\
	1&2x&1&\cdots&0&0&0\\
	0&1&2x&\cdots&0&0&0\\
	\vdots&\vdots&\vdots&\ddots&\vdots&\vdots&\vdots\\
	0&0&0&\cdots&2x&1&0\\
	0&0&0&\cdots&1&2x&2\\
	0&0&0&\cdots&0&1&2x
\end{pmatrix}
}
\]
and
\[
U_n\left(x\right)
=
\det{
\begin{pmatrix}
	2x&1&0&\cdots&0&0&0\\
	1&2x&1&\cdots&0&0&0\\
	0&1&2x&\cdots&0&0&0\\
	\vdots&\vdots&\vdots&\ddots&\vdots&\vdots&\vdots\\
	0&0&0&\cdots&2x&1&0\\
	0&0&0&\cdots&1&2x&1\\
	0&0&0&\cdots&0&1&2x
\end{pmatrix}
}.
\]
The first few polynomials are
\bd
\begin{array}{lcl}
	T_0(x)  & = &  1 \\
	T_1(x)  & = &  x \\
	T_2(x)  & = & 2x^2-1 \\
	T_3(x)  & = & 4x^3-3x \\
	T_ 4(x) & = & 8 x^4-8x^2+1 \\
	T_ 5(x) & = & 16x^5-20x^3+5x \\
	T_ 6(x) & = & 32x^6-48x^4+18 x^2 -1,
\end{array}
\ed
and
\bd
\begin{array}{lcl}
	U_0(x) & = & 1 \\
	U_1(x) & = & 2x \\
	U_2(x) & = & 4x^2-1 \\
	U_3(x) & = & 8x^3-4x \\
	U_4(x) & = & 16 x^4-12x^2+1 \\
	U_5(x) & = & 32x^5-32x^3+6x \\
	U_6(x) & = & 64x^6-80x^4+24 x^2 -1.
\end{array}
\ed
Expanding the determinants with respect to the first row we obtain the recurrence
\[F_{n+1}=2xF_n-F_{n-1}.\]
For the initial values $F_0=1,\;F_1=x$ and $F_0=1,\;F_1=2x$ we obtain the Chebyshev
polynomials of first and second kind respectively.

For $x=\cos \theta$, the trigonometric identities
\[
2x \cos{n\theta}=\cos{(n+1)\theta} + \cos{(n-1)\theta}
\]
and
\[
2x \sin{(n+1)\theta}=\sin{(n+2)\theta} + \cos{n\theta}
\]
give
\[
T_n(x)=\cos{n\theta}, \ U_n(x)=\frac{\sin{(n+1)\theta}}{\sin{\theta}}.
\]

We list some properties of Chebyshev polynomials that can be easily proved by induction.

\noindent
The polynomials $T_n$ satisfy:
\bd
\begin{array}{lcl}
	T_n(-x)&=&(-1)^n T_n(x)\\
	T_n(1)&=&1 \\
	T_{2n}(0)&=&(-1)^n\\
	T_{2n-1}(0)&=&0,
\end{array}
 \ed
while polynomials $U_n$ satisfy:
\bd
\begin{array}{lcl}
	U_n(-x)&=&(-1)^n U_n(x)\\
	U_n(1)&=&n+1 \\
	U_{2n}(0)&=&(-1)^n\\
	U_{2n-1}(0)&=&0.
\end{array}
 \ed
In addition (see \cite{Mason}%
)
\bd
	T_n(x)=2^{n-1} \prod_{j=1}^n \left[ x-\cos \left( \frac{(2j-1) \pi}{2n}\right) \right]
\ed
and
\bd
	U_n(x)=2^n \prod_{j=1}^n \left[ x-\cos \left( \frac{j\pi}{n+1}\right) \right].
\ed

There are some explicit formulas in powers of $x$ which are used in propositions
\ref{charant}, \ref{charbnt}, \ref{charcnt} and \ref{chardnt}.
\be
\label{Tproduct}
	T_n(x)=n \sum_{j=0}^n (-2)^j \frac{ (n+j-1)!}{ (n-j)! (2j)!} (1-x)^j    \ \ (n>0),
\ee
and
\be
\label{Uproduct}
	U_n(x)=\sum_{j=0}^n (-2)^j  \binom{ n+j+1}{ 2j+1} (1-x)^j.
\ee

\subsection{Cyclotomic Polynomials}

A complex number $\omega$ of order $n$ is called a  primitive $n$\ts{th} root of unity,
e.g. $e^{\frac{2 \pi i}{n}}$ is a primitive $n$\ts{th} root of unity. If $\omega$ is a
primitive $n$\ts{th} root of unity then $\omega^k$ is a primitive $n$\ts{th} root of unity
if and only if $\gcd(n,k)=1$. Since $e^{\frac{2 \pi i}{n}}$ produces all $n$\ts{th} roots of unity
it follows that there are exactly $\phi(n)$ primitive $n$\ts{th} roots of unity where $\phi$ is
Euler's totient function. Primitive $n$\ts{th} roots of unity are conjugate algebraic integers.

Let $\omega$ be a primitive $n$\ts{th} root of unity and $\Phi_n(x)$ its minimal polynomial. Then
\begin{displaymath}
	\Phi_n(x)=(x-\omega^{k_1})(x-\omega^{k_2}) \cdots (x-\omega^{k_{\phi(n)}}),
\end{displaymath}
where $1\leq k_1, k_2, \ldots, k_{\phi(n)} < n$ are the integers relatively prime to $n$.
The polynomial $\Phi_n(x) \in\mathbb Z[x]$ is the $n\ts{th}$ cyclotomic polynomial.
From the relation $\omega^{k\cdot n}-1=0$ it follows that $\Phi_n(x)|x^{k\cdot n}-1, \ \forall k\in\mathbb N$.
Actually
\[
x^n-1=\prod_{d|n}\Phi_d(x).
\]

Following  Lehmer \cite{Lehmer}, using cyclotomic polynomials
we can derive the minimal polynomials $\Psi_n$ of the algebraic integers
$2 \cos{\frac{2 k \pi}{n}}$, where $\gcd{(k,n)=1}$ (for $n\geq 2$). The polynomial $\Phi_n$,
being reciprocal (i.e. $\Phi_n(x)=x^{\phi(n)}\Phi_n\left(\frac{1}{x}\right)$),
it can be written in the form
\be
\label{Psi_n}
	\Phi_n(x)=x^{\frac{\phi(n)}{2}}\Psi_n\left(x+\frac{1}{x}\right),
\ee
for some monic irreducible polynomial $\Psi_n$ with integer coefficients
and degree half of that of $\Phi_n$. The
irreducibility of $\Psi_n$ is equivalent to the irreducibility of $\Phi_n$.

For $x=e^{\frac{2 k\pi i}{n}}$, a primitive $n$\ts{th} root of unity, we have that
\[
x+\frac{1}{x}=2 \cos{\frac{2 k \pi}{n}}.
\]
From equation \ref{Psi_n} we deduce that $2 \cos{\frac{2 k \pi}{n}}$ is a root of the irreducible
polynomial $\Psi_n$. Therefore $\Psi_n$ is the minimal polynomial of $2 \cos{\frac{2 k \pi}{n}}$.
Also equation \ref{Psi_n} can be used for the calculation of the polynomials $\Psi_n$.
For example, $\Psi_5(x)  = x^2+x-1$, $\Psi_9(x)  = x^3-3x+1$.

The roots of the polynomials $\Psi_n$ are
\[
2 \cos{\frac{2 k \pi}{n}},\; \ \text{where}\;\gcd(k,n)=1.
\]

\section{Affine Lie algebras}

\subsection{Cartan matrix of type \texorpdfstring{$A_n^{(1)}$}{An} }

The Cartan matrix of type $A_n^{(1)}$ is a matrix of the form
\be
\label{cartanant}
C_{A_n^{(1)}}=
\begin{pmatrix}
	2 & -1& 0 & \cdots &0 &0 & -1 \cr
	-1& 2 &-1 & \cdots &0 &0 & 0 \cr
	0 &-1 & 2 & \cdots &0 &0 & 0 \cr
	\vdots  &\vdots & \vdots & \ddots &\vdots &\vdots &\vdots \cr
	0 & 0 & 0 & \cdots &2 &-1 & 0 \cr
	0 & 0 & 0 & \cdots &-1& 2 &-1  \cr
	-1 & 0 & 0 & \cdots &0 &-1 & 2
\end{pmatrix}.
\ee

We list some formulas for the characteristic polynomial of the matrix  for
small values of $n$:
\bd
\begin{array}{rcl}
	p_3(x) & = & x^3-6 x^2+9x                             = x(x-3)^2\\
	p_4(x) & = & x^4-8x^3+20x^2-16x                       = x(x-4)(x-2)^2 \\
	p_5(x) & = & x^5-10 x^4+35 x^3 -50 x^2 +25x           = x(x^2-5x+5)^2 \\
	p_6(x) & = & x^6-12 x^5+54 x^4-112 x^3 +105 x^2 -36x  = x(x-4)(x-1)^2(x-3)^2 \\
	p_7(x) & = & x^7-14x^6+77x^5-210x^4+294x^3-196x^2+49x = \\
	& & x(x^3-7x^2+14x-7)^2.
\end{array}
\ed

We define a sequence of  polynomials in  the following way
\bd
q_n(x)=
\det
\begin{pmatrix}
	2x&  1& 0 & \cdots &0 &0 & 1 \\
	1 & 2x& 1 & \cdots &0 &0 & 0 \\
	0 & 1 & 2x& \cdots &0 &0 & 0 \\
	\vdots  &\vdots & \vdots & \ddots &\vdots &\vdots &\vdots \\
	0 & 0 & 0 & \cdots &2x& 1 & 0  \\
	0 & 0 & 0 & \cdots & 1& 2x& 1  \\
	1 & 0 & 0 & \cdots &0 & 1 & 2x
\end{pmatrix}.
\ed
By expanding the determinant we obtain the following formula for $q_n$
\begin{gather*}
	q_{n}(x)=2x U_{n-1}(x)-2 U_{n-2}(x)+2 (-1)^{n-1}=U_n(x)-U_{n-2}(x)+2(-1)^{n-1} =\\
	2 T_n(x)+2(-1)^{n-1}.
\end{gather*}
It is easy then to compute  the first few polynomials:
\bd
\begin{array}{lcl}
	q_3(x) & = & 8x^3-6x+2                   = 2(x+1)(2x-1)^2               \\
	q_4(x) & = & 16 x^4-16x^2                = 16x^2(x-1)(x+1)              \\
	q_5(x) & = & 32x^5-40x^3+10x+2           = 2(x+1)(4x^2-2x-1)^2          \\
	q_6(x) & = & 64x^6-96x^4+36 x^2-4        = 4(x+1)(x-1)(2x-1)^2(2x+1)^2  \\
	q_7(x) & = & 128 x^7-224x^5+112x^3-14x+2 = 2(x+1)(8x^3-4x^2-4x+1)^2     .
\end{array}
 \ed
For $n$ even the polynomial $q_n$ is divisible by $x-1$. Indeed
\bd
	q_n(1)=U_n(1)-U_{n-2}(1)+2 (-1)^{n-1}=(n+1)-(n-1)-2=0.
\ed

\begin{remark}
Note that
\begin{gather*}
	p_n(0)=q_n(-1)=U_n(-1)- U_{n-2}(-1)+2(-1)^{n-1} =\\
	(-1)^n U_n(1)- (-1)^{n-2} U_{n-2}(1)+2(-1)^{n-1} =0.
\end{gather*}
Therefore the determinant of $A_n^{(1)}$ is zero and $p_n$ is divisible by $x$.

\end{remark}

\begin{proposition}
\label{charant}
Let $p_n$ be the characteristic polynomial of the Cartan matrix (\ref{cartanant}).
Then
\bd
	p_n(x)= \sum_{j=1}^{n}  (-1)^{n+j} \frac{ 2n  (n+j-1)!}{ (n-j)! (2j)!}  x^j.
\ed

\end{proposition}

\begin{proof}
Using the properties of Chebyshev polynomials it follows that
\bd
	\frac{1}{2} p_n(x)+(-1)^{n} =
	\frac{1}{2} q_n\left(\frac{x}{2}-1\right) +(-1)^{n} =
	T_n \left( \frac{x}{2}-1 \right)=
	(-1)^{n}T_n \left(1- \frac{x}{2} \right).
\ed
Using equation \ref{Tproduct} we have
\begin{gather*}
	T_n \left(1- \frac{x}{2} \right)=
	n \sum_{j=0}^n (-2)^j \frac{ (n+j-1)!}{ (n-j)! (2j)!}
	\left(1-\left(1-\frac{x}{2} \right) \right)^j=\\
	n \sum_{j=0}^n (-1)^j \frac{ (n+j-1)!}{ (n-j)! (2j)!} x^j.
\end{gather*}
Therefore
\[
	p_n(x)=\sum_{j=0}^{n}  (-1)^{n+j} \frac{ 2n  (n+j-1)!}{ (n-j)! (2j)!} x^j +2(-1)^{n-1}
\]
and the result follows.
\end{proof}

\subsection{Cartan matrix of type \texorpdfstring{$B_n^{(1)}$}{Bn} }

The Cartan matrix of type $B_n^{(1)}$ is a matrix of the form
\be
\label{cartanbnt}
C_{ B_n^{(1)}}=
\begin{pmatrix}
	2 & 0 &-1 & 0 &\cdots &0 &0 & 0 \\
	0 & 2 &-1 & 0 &\cdots &0 &0 & 0 \\
	-1 &-1 & 2 &-1 &\cdots &0 &0 & 0 \\
	0 & 0 &-1 & 2 &\cdots &0 &0 & 0 \\
	\vdots  &\vdots & \vdots & \vdots & \ddots &\vdots &\vdots &\vdots \\
	0 & 0 & 0 & 0 &\cdots &2 &-1 & 0 \\
	0 & 0 & 0 & 0 &\cdots &-1& 2 &-2 \\
	0 & 0 & 0 & 0 &\cdots &0 &-1& 2
\end{pmatrix}.
\ee
Using expansion on the first row it is easy to prove that
$\det\left(C_{ B_n^{(1)}}\right)=0$.

We list some formulas for the characteristic polynomial of the matrix  for small values of $n$:
\bd
\begin{array}{rcl}
	p_4(x) & = & x^4-8 x^3+20 x^2-16 x                         = x(x-4)(x-2)^2         \\
	p_5(x) & = & x^5-10 x^4+35 x^3 -50 x^2 +24 x               = x(x-1)(x-2)(x-3)(x-4) \\
	p_6(x) & = & x^6-12 x^5+54 x^4-112 x^3 +104 x^2 -32 x      =                       \\
	& & x(x-4)(x-2)^2(x^2-4 x+2)                                                       \\
	p_7(x) & = & x^7-14 x^6+77 x^5-210 x^4+293 x^3-190 x^2+40 x =                      \\
	& & x(x-2)(x-4)(x^2-5 x+5)(x^2-3 x+1).
\end{array}
\ed

Define the polynomial $q_n(x)=\det{(2xI+A)}$.
By expanding the determinant we obtain the following formula for $q_n$
\bd
	q_{n}(x)=4x\left(T_{n-1}(x)- T_{n-3}(x)\right) =8x\left(x^2-1\right)U_{n-3}(x).
\ed
Equivalently
\bd
	q_n(x)= 2 \left(T_n(x)- T_{n-4}(x) \right).
\ed
The first few polynomials are:
\bd
\begin{array}{lcl}
	q_4(x) & = & 16 x^4-16x^2             = 16x^2(x-1)(x+1)                    \\
	q_5(x) & = & 32x^5-40x^3+8x           = 8x(x-1)(2x+1)(2x-1)(x+1)           \\
	q_6(x) & = & 64x^6-96x^4+32 x^2       = 32x^2(x-1)(x+1)(2x^2-1)            \\
	q_7(x) & = & 128 x^7-224x^5+104x^3-8x =                                    \\
	& & 8x(x-1)(x+1)(4x^2-2x-1)(4x^2+2x-1).
\end{array}
\ed
We can easily compute the explicit form of the $p_n$ polynomial.
\begin{proposition} \label{charbnt}
Let $p_n(x)$ be the characteristic polynomial of the Cartan matrix (\ref{cartanbnt}).
Then
\bd
	p_n(x)=x(x-2)(x-4) \sum_{j=0}^{n-3} (-1)^{n+j+1}  \binom{ n+j-2}{ 2j+1} x^j.
\ed
\end{proposition}

\begin{proof}
From $p_n(x)=q_n\left(\frac{x}{2}-1\right)$ and $q_{n}(x)=8x\left(x^2-1\right)U_{n-3}(x)$ we
only need to show that
\[
U_{n-3}\left(\frac{x}{2}-1 \right) = \sum_{j=0}^{n-3} (-1)^{n+j+1} \binom{ n+j-2}{ 2j+1} x^j,
\]
which is equation \ref{Uproduct} combined
with $U_n(-x)=(-1)^nU_n(x).$
\end{proof}

\subsection{Cartan matrix of type \texorpdfstring{$C_n^{(1)}$}{Cn} }

The Cartan matrix of type $C_n^{(1)}$ is a tri-diagonal matrix of the form
\be
\label{cartancnt}
C_{C_n^{(1)}} =
\begin{pmatrix}
	2 & -2& 0 & \cdots &0 &0 & 0 \\
	-1& 2 &-1 & \cdots &0 &0 & 0 \\
	0 &-1 & 2 & \cdots &0 &0 & 0 \\
	\vdots  &\vdots & \vdots & \ddots &\vdots &\vdots &\vdots \\
	0 & 0 & 0 & \cdots &2 &-1 & 0 \\
	0 & 0 & 0 & \cdots &-1& 2 &-1 \\
	0 & 0 & 0 & \cdots &0 &-2 & 2
\end{pmatrix}.
\ee
Using expansion on the first row it is easy to prove that  $\det\left( C_{C_n^{(1)}} \right)=0$.

We list the formula for the characteristic polynomial of the matrix  for small values of $n$:
\bd
\begin{array}{rcl}
	p_3(x) & = & x^3-6 x^2+8 x                                  = x(x-2)(x-4)            \\
	p_4(x) & = & x^4-8 x^3+19 x^2-12 x                          = x(x-1)(x-3)(x-4)       \\
	p_5(x) & = & x^5-10 x^4+34 x^3 -44 x^2 +16 x                = x(x-2)(x-4)(x^2-4 x+2) \\
	p_6(x) & = & x^6-12 x^5+53 x^4-104 x^3 +85 x^2 -20 x        =                        \\
	& & x(x-4)(x^2-5 x+5)(x^2-3 x+1)                                                     \\
	p_7(x) & = & x^7-14 x^6+76 x^5-200 x^4+259 x^3-146 x^2+24 x =                        \\
	& & x(x-1)(x-2)(x-3)(x-4)(x^2-4 x +1).
\end{array}
\ed
Define $q_n(x)=\det{\left(2xI+A\right)}$.
By expanding the determinant with respect to the first row we obtain
\bd
	q_n(x)=4\left(x T_{n-1}(x)-T_{n-2}(x)\right),
\ed
where $T_n(x)$ is the $n\ts{th}$ Chebyshev polynomial of the first kind. Equivalently,
\bd
	q_n(x)=2\left(T_n(x)- T_{n-2}(x)\right)=4\left(x^2-1\right)U_{n-2}(x).
\ed

\begin{proposition} \label{charcnt}
Let $p_n(x)$ be the characteristic polynomial of the Cartan matrix (\ref{cartancnt}).  Then
\bd
	p_n(x)= x(x-4) \sum_{j=0}^{n-2}  (-1)^{n+j}  \binom{ n+j-1}{ 2j+1} x^j.
\ed
\end{proposition}

\subsection{Cartan matrix of type \texorpdfstring{$D_n^{(1)}$}{Dn} }
The Cartan matrix of type $  D_n^{(1)} $ is a    matrix of the form
\be
\label{cartandnt}
C_{D_n^{(1)}}=
\begin{pmatrix}
	2 & 0 &-1 & 0 &\cdots &0 &0 &0 & 0 \cr
	0 & 2 &-1 & 0 &\cdots &0 &0 &0 & 0 \cr
	-1&-1 & 2 &-1 &\cdots &0 &0 &0 & 0 \cr
	0 & 0 &-1 & 2 &\cdots &0 &0 &0 & 0 \cr
	\vdots  &\vdots & \vdots & \vdots & \ddots &\vdots &\vdots &\vdots &\vdots\cr
	0 & 0 & 0 & 0 &\cdots & 2 &-1 & 0 & 0 \cr
	0 & 0 & 0 & 0 &\cdots &-1 &2 &-1 &-1 \cr
	0 & 0 & 0 & 0 &\cdots &0 &-1& 2 & 0  \cr
	0 & 0 & 0 & 0 &\cdots &0 &-1& 0 & 2
\end{pmatrix}.
\ee

We list the formula for the characteristic polynomial
of the matrix $C_{D_n^{(1)}}$ for small values of $n$:
\bd
\begin{array}{rcl}
	p_5(x)    &=& x(x-4)(x-2)^3                                     \\
	p_6(x)    &=& x(x-1)(x-2)^2 (x-3)(x-4)                          \\
	p_7(x)    &=& x(x-2)^3(x-4)(x^2-4x+2)                           \\
	p_8(x)    &=& x(x-2)^2 (x-4)(x^2-3x+1)(x^2-5x+5)                \\
	p_9(x)    &=& x(x-2)^3(x-4) (x-1)(x-3) (x^2-4x+1)                .
\end{array}
\ed
By expanding the determinant $q_n(x)=\det{(2xI+A)}$ with respect to the first row we conclude that
\bd
	q_n(x)=2x \hat{q}_{n-1} (x)- 2x \hat{q}_{n-3} (x),
\ed
where $\hat{q}_n$ is the $q_n$ polynomial of the matrix $C_{D_n}$ (see \cite{Damianou2014}).
Therefore
\bd
	q_n(x)= 8 x^2( T_{n-2}(x)-T_{n-4}(x) ) =16x^2(x^2-1)U_{n-4}(x).
\ed
\begin{proposition} \label{chardnt}
Let $p_n(x)$ be the characteristic polynomial of the Cartan matrix (\ref{cartandnt}).  Then
\bd
	p_n(x)= x(x-2)^2(x-4) \sum_{j=0}^{n-4}  (-1)^{n+j}  \binom{ n+j-3}{ 2j+1} x^j.
\ed
\end{proposition}

\section{Coxeter Systems}
\subsection{Coxeter Polynomials}
A Coxeter group $W$ is a group with presentation
\be \label{Coxetergroup}
	W=\left\langle w_1, w_2, \dots, w_n \ | \ w_i^2=1, (w_i w_j)^{m_{ij}} =1 \right\rangle,
\ee
where  $m_{ij} \in \left\{ 3, 4, \dots, \infty\right\}$. The pair $(W,S)$ is called a
Coxeter system where  $S=\left\{ w_i \ | \ i=1, \dots, n \right\} $. The Coxeter matrix
is the $n\times n$ matrix with elements $m_{ij}$, which is usually encoded in the Coxeter graph.
This graph is a simple graph $\Gamma$ with $n$ vertices
and edge weights $m_{ij} \in \left\{ 4, 5, \ldots, \infty\right\}$.
If $m_{ij}=2$, the vertex $i$ is not connected with the vertex $j$
while if $m_{ij}=3$ there is an unweighted edge between the vertices $i$ and $j$.
Coxeter groups can be visualized as groups of reflections.
We define a bilinear form $B$ on $\mathbb{R}^n$ by choosing a
basis $e_1, e_2, \dots, e_n$ and setting
\bd
	B(e_i, e_j)=-2 \cos { \frac{\pi}{m_{ij}} }.
\ed
If $m_{ij}=\infty$ we define $B(e_i, e_j)=-2$.
We also define for $i=1,2, \dots, n$ the reflection
\bd
	w_i (e_j)=e_j - B(e_i, e_j) e_i.
\ed
The group generated by $w_i$ is a Coxeter group with presentation (\ref{Coxetergroup}).
It is well-known that $W$ is finite if and only if $B$ is positive definite.
 A \textit{Coxeter element} (or transformation)  is a product of the form
\bd
	w_{\pi(1)} w_{\pi(2)} \dots w_{\pi(n)} ,  \ \pi \in S_n.
\ed
If the Coxeter graph $\Gamma$ is a tree then the Coxeter elements are in
a single conjugacy class in $W$. A \textit{Coxeter polynomial}
for the Coxeter system $(W,S)$ is the characteristic polynomial of a Coxeter element.
For Coxeter systems whose graphs are trees the Coxeter polynomial is uniquely defined.
This covers all the cases we investigate except $A_n^{(1)}$.
This case was examined in \cite{Coleman}, where it was shown
that there are $\left\lfloor\frac{n+1}{2}\right\rfloor$ non-conjugate Coxeter elements.

Weyl groups of Lie algebras are important examples of Coxeter groups.
For a Lie algebra $\g$ with simple roots
$\Pi=\left\{\ga_1,\ga_2,\ldots, \ga_{\ell}\right\}$ and Cartan matrix $C$
the Weyl group is generated by the reflections
$\sigma_j$, on the real vector space $V=\operatorname{span}\{\ga_1,\ldots,\ga_n\}$,
defined by $\sigma_j(\ga_i)=\ga_i-C_{i,j}\ga_j$.
The Weyl group is a Coxeter group with presentation:
\begin{enumerate}
	\item $\sigma_i^2=1$
	\item $(\sigma_i\sigma_j)^2=1$ ,  if  $C_{i,j}C_{j,i}=0$
	\item $(\sigma_i\sigma_j)^3=1$ ,  if $C_{i,j}C_{j,i}=1$
	\item $(\sigma_i\sigma_j)^{2k}=1$ , if  $C_{i,j}C_{j,i}=k \in \{2,3\}$.
\end{enumerate}
This presentation can be seen if we represent the reflections
$\sigma_j$ with respect to the basis
$
\left\{
\frac{\ga_1}{\|\ga_1\|},\frac{\ga_2}{\|\ga_2\|},\cdots, \frac{\ga_\ell}{\|\ga_\ell\|}
\right\}
$.
Therefore from the Dynkin diagram we obtain the corresponding Coxeter graph $\Gamma $ as follows.

\noindent
The Coxeter graph $\Gamma$ has the same vertex set as the Dynkin diagram
and there is an edge between the vertices $\ga_i,\ga_j$ if and only if $C_{i,j}\neq 0$.
The edge $(\ga_i,\ga_j)$ is labelled with $2k$ whenever $C_{i,j}C_{j,i}=k$, for $k \in \{ 2,3 \}$.
Note that for the simply laced cases the Dynkin diagram coincides with the Coxeter graph.

\begin{remark}
From the presentation of the Weyl group as a Coxeter group we see
that the Coxeter group of a twisted affine Dynkin diagram coincides with
the Coxeter group of an untwisted affine Dynkin diagram.
\end{remark}

\begin{example}

Consider the Dynkin diagram of type $B^{(1)}_4$.

\hspace{1 ex}


$
	\put(25,42){\line(5,-2){50}}
	\put(25,-2){\line(5,2){50}}
	\put(85,20){\line(1,0){50}}
	\put(145,19){\line(1,0){50}}
	\put(145,21){\line(1,0){50}}
	\put(20.2,43.3){\circle{10}}
	\put(20.2,-3.3){\circle*{10}}
	\put(80,20){\circle{10}}
	\put(140,20){\circle{10}}
	\put(200,20){\circle{10}}
	\put(0,50){$\boldsymbol{B_{4}^{(1)}}$}
	\put(167,17.5){$ > $}
$


\hspace{1 ex}

\noindent
The corresponding Coxeter graph is

\hspace{1 ex}

$
	\put(25,42){\line(5,-2){50}}
	\put(25,-2){\line(5,2){50}}
	\put(85,20){\line(1,0){50}}
	\put(145,20){\line(1,0){50}}
	\put(20.2,43.3){\circle{10}}
	\put(20.2,-3.3){\circle{10}}
	\put(80,20){\circle{10}}
	\put(140,20){\circle{10}}
	\put(170,25){4}
	\put(200,20){\circle{10}}
$


\hspace{1 ex}

\noindent
and the bilinear form is defined by the Coxeter matrix


\bd
\begin{pmatrix}
    1 & 2 & 3 & 2 & 2 \\
    2 & 1 & 3 & 2 & 2 \\
    3 & 3 & 1 & 3 & 2 \\
    2 & 2 & 3 & 1 & 4 \\
    2 & 2 & 2 & 4 & 1
\end{pmatrix} .
\ed
The reflection $\sigma_1$  is defined by
\bd
	\sigma_1 (e_1)=-e_1, \ \
	\sigma_1 (e_2)=e_2, \ \
	\sigma_1(e_3)=e_3+e_1, \ \
	\sigma_1(e_4)=e_4, \ \
	\sigma_1(e_5)=e_5,
\ed
and it has matrix representation

\bd
\sigma_1=
\begin{pmatrix}
    -1& 0 & 1 & 0 & 0      \\
    0 & 1 & 0 & 0 & 0      \\
    0 & 0 & 1 & 0 & 0      \\
    0 & 0 & 0 & 1 & 0      \\
	0 & 0 & 0 & 0 & 1
\end{pmatrix}.
\ed
Similarly we define the other simple reflections $\sigma_i$
and the Coxeter element is
\bd
w=\sigma_1 \sigma_2 \sigma_3 \sigma_4 \sigma_5=
\begin{pmatrix}
	0 & 1 & 0 & 1 & -\sqrt{2} \\
    1 & 0 & 0 & 1 & -\sqrt{2} \\
    1 & 1 & 0 & 1 & -\sqrt{2} \\
    0 & 0 & 1 & 1 & -\sqrt{2} \\
    0 & 0 & 0 &\sqrt{2}& -1
\end{pmatrix}.
\ed
The Coxeter polynomial is the characteristic polynomial of $w$
\bd
	f(x)= x^5-x^3-x^2+1=x^3(x^2-1)-(x^2-1)=(x-1)^2(x+1)(x^2+x+1)\ .
\ed

Note that the roots of $f$ are of the form
\bd
	e^{\frac{2m_j \pi i }{h}},
\ed
where $h=6$ and $m_j$ take the values $0,2,3,4,6$.
\end{example}

\subsection{Affine Coxeter number and Affine Exponents} \label{affCoxeternumber}

Let us recall the definition of exponents for a simple complex Lie group
$G$, see \cite{Bourbakifourtosix}, \cite{Coleman},  \cite{Kostant}.  Suppose  $G$ is  a connected, complex,
 simple Lie Group $G$.  We
form the de Rham cohomology groups $H^i (G, \mathbb{C})$ and the
corresponding Poincar\'e polynomial of $G$:
\bd
	p_G(t)=\sum_{i=1}^{\ell}  b_i t^i,
\ed
 where $b_i=  $dim  $H^i (G, \mathbb{C})$ are the Betti numbers of $G$. The De Rham groups
 encode topological information about $G$.
 Following work of Cartan, Pontryagin and Brauer,   Hopf proved that the
cohomology algebra is  isomorphic to that of  a finite product of $\ell$
 spheres of odd dimension  where $\ell$ is the rank of $G$. This result  implies that
\bd
	p_G(t)= \prod_{i=1}^{\ell} (1+t^{2 m_i^{\prime} +1} ).
\ed
The  positive integers $\{ m_1^{\prime}, m_2^{\prime}, .... , m_{\ell}^{\prime} \}$ are called the
{\it exponents} of $G$. They are also the exponents of the Lie algebra $\g$ of
$G$. The roots of the Coxeter polynomial of $\g$ are
\bd
	e^{\frac{2 m_j^{\prime} \pi i}{h^{\prime} }}, \ i=1,2,\ldots,\ell
\ed
where $h^{\prime}$ is the Coxeter number (see \cite{Coxeter2}).

For an affine Lie algebra with affine Cartan matrix $C$ of rank $n$,
the roots of the Coxeter polynomial $f(x)$ are in the unit disk. Thus from a theorem of
Kronecker $f(x)$ is a product of cyclotomic polynomials (see \cite{Damianou}).
Let $V$ be the real vector space $\operatorname{span}\left\{\ga_0,\ga_1,\ldots,\ga_n\right\}$
and $D$ a diagonal matrix with positive
entries such that $CD$ is symmetric. The matrix $CD$ defines a semi-positive
bilinear form $(\, ,)$ on $V$. For $\ga=\sum_{i=0}^{n} z_i\ga_i$,
$\ (\ga,.)=0 \iff \left(z_0,z_1,\ldots,z_n\right)C=0$. Thus if
$\left(z_0,z_1,\ldots,z_n\right)$ is a left zero eigenvector of $C$
and $\ga=\sum_{i=0}^{n} z_i\ga_i$, the induced bilinear form on the vector space
$\tilde{V}=V/\langle \ga\rangle$ is positive definite and corresponds to
an $n\times n$ submatrix of $C$. Therefore if $\sigma:V\longrightarrow V$
is a Coxeter transformation of the affine Lie algebra, the induced transformation
on $\tilde{V}$ has finite order $h$. It follows that $(\sigma-1)(\sigma^h-1)=0$
and therefore the roots of $f$ are $h^\ts{th}$ roots of unity.

The roots of $f(x)$ are
\be
\label{RootsOfCoxPol}
	\left\{e^{\frac{2 m_j \pi i}{h}}: m_j\in \{0,1,\ldots,h\}\right\}.
\ee
We define the integers $m_j$ to be the affine exponents and $h$ the
affine Coxeter number associated with the Coxeter transformation $\sigma$.
These numbers are uniquely defined for each affine Lie algebra
except in the case of $A_n^{(1)}$ where we define them for each conjugacy class.
Using the fact that the corank of the Cartan matrix $C_{X^{(1)}}$ is $1$
and the relation between the polynomials $p(x)$ and $f(x)$ (for the bipartite case)
it follows that $(x-1)^2 \mid f(x)$ and $(x-1)^3\nmid f(x)$. For the factor $(x-1)^2$
we define the associated affine exponents to be $0$ and $h$.

\begin{definition}\label{Branch vertex definition}
For a finite Dynkin diagram $\Gamma$ with corresponding Cartan matrix $C$ we define a weight
function $b:V(\Gamma)\rightarrow\mathbb N$ on the vertices of $\Gamma$. If the vertex $r_i$
has only one neighbor we define $b(r_i)=1-\sum_{j\neq i}C_{i,j}$ while
if it has more than one neighbors we define $b(r_i)=-\sum_{j\neq i}C_{i,j}$.
For a Dynkin diagram $\Gamma$ not of type $A_n$, we define the branch vertex $r_i$ to be the one
which maximize $b$. For the case of $A_n, \ n$ odd, we define the branch vertex to be
the middle one.
\end{definition}
\begin{example}
For the case of the Dynkin diagrams $D_n,E_6,E_7,E_8$ the branch vertex is the one which
is the common endpoint of three edges. For the Dynkin diagram of type $C_n$ the branch vertex
is the one which corresponds to the highest root.
\end{example}
From Steinberg's theorem \cite[p.591 ]{Steinberg} we see that,
the affine exponents and the affine Coxeter number of an affine
Lie algebra of type $X_n^{(1)}\neq A_n^{(1)}$ can be computed using  the exponents
and Coxeter number of  the root system $A_n$. Let
$\Pi=\{\ga_1,\ga_2,\ldots, \ga_{n}\}$ be the simple roots of the Lie algebra of type $X_n$.
Define the branch root $\beta$ to be that root
which corresponds to the branch vertex of the corresponding Dynkin diagram.
If we delete the branch root the reduced system is a product of
root systems of type $A_m$. The Coxeter polynomial of $X_n^{(1)}$ is
$f(x)=(x-1)^2g(x)$ with $g(x)$ the Coxeter polynomial
of the reduced system. The affine Coxeter number $h$ is the
Coxeter number of the reduced system and the
affine exponents are obtained using the following procedure:

\noindent
From the factor $(x-1)^2 \mid f(x)$ it follows that $0$ and $h$ are affine exponents.
If $Y$ appears in the reduced system and $m_j^{'}$ is an exponent
and $ h^{'}$  is the Coxeter number of $Y$ then $\dfrac{h}{h^{'}}m_j^{'}$
is an affine exponent of $X^{(1)}$.

\begin{table}[!htbp]
  \caption{Affine Exponents and affine Coxeter number for affine root systems.
  \label{affinecoxetern}}
  \begin{center}
  \begin{tabular}{|c|c|c|}
    \hline
       \ Root system\ \ & Affine Exponents\ \ & Affine Coxeter number\\
    \hline
      $A_{n}^{(1)}    $ & $0,k_j,2k_j,\ldots,jk_j,           $ & $jk_j$    \\
                        & $n_j,2n_j,\ldots,(n-j)n_j          $ &           \\
      $B_{2n+1}^{(1)} $ & $0,1,2,3,\ldots,2n,n               $ & $2n$      \\
      $B_{2n}^{(1)}   $ & $0,2,4,\ldots,2(2n-1),2n-1         $ & $2(2n-1)$ \\
      $C_n^{(1)}      $ & $0,1,2,\dots, n                    $ & $n$       \\
      $D_{2n+1}^{(1)} $ & $0,2,4,\ldots,2(2n-1),2n-1,2n-1    $ & $2(2n-1)$ \\
      $D_{2n}^{(1)}   $ & $0,1,2,3,\ldots,2n-2,n-1,n-1       $ & $2n-2$    \\
      $E_6^{(1)}      $ & $0,2,2,3,4,4,6                     $ & $6$       \\
      $E_7^{(1)}      $ & $0,3,4,6,6,8,9,12                  $ & $12$      \\
      $E_8^{(1)}      $ & $0,6,10,12,15,18,20,24,30          $ & $30$      \\
      $F_4^{(1)}      $ & $0,2,3,4,6                         $ & $6 $      \\
      $G_2^{(1)}      $ & $0,1,2                             $ & $2$       \\
    \hline
  \end{tabular}
  \end{center}
\end{table}

\begin{example}
For the root system $E_8^{(1)}$ the reduced system is $A_1 \times A_2 \times A_4$.
The exponents of $A_n$ are $1,2,\ldots,n$ and the Coxeter number
$n+1$ (see \cite{Damianou2014}). Therefore the affine Coxeter number for the root system $E_8^{(1)}$
is $lcm(2,3,5)=30$ and the affine exponents are $0, 6, 10, 12, 15, 18, 20, 24, 30$.
\end{example}

In table \ref{affinecoxetern} we list the affine exponents and the affine Coxeter number for
the affine Lie algebras. In the case of $A_n^{(1)}$, for
$j=1,2,\ldots,\left\lfloor\frac{n+1}{2}\right\rfloor$ we have denoted $k_j=\frac{n+1-j}{d_j}$
and $n_j=\frac{j}{d_j}$, where $d_j=\gcd(n+1,j)$. The affine exponents and
affine Coxeter number of $A_n^{(1)}$ given in table \ref{affinecoxetern}, are those
associated with the Coxeter polynomial $(x^j-1)(x^{n+1-j}-1)$. For $n$ odd and
$j=\frac{n+1}{2}$ we have $n_j=k_j=1, d_j=j$ and we obtain the case considered in \cite{Moody}.
Note the duality in the set of affine exponents:
\be  \label{dual}
	m_i+m_{n-i} = h, \ i=0,1,\ldots, \left\lfloor\frac{n}{2}\right\rfloor,
\ee
where $h$ is the affine Coxeter number. This is a consequence of (\ref{RootsOfCoxPol}).

The affine exponents, affine Coxeter number of $X_n^{(1)}$ and the roots of $X_n$ are related in
a mysterious way (see \cite{Moody}). Let $\Pi=\{\ga_1,\ga_2,\ldots,
\ga_{n}\}$ be the simple roots of the Lie algebra of type $X_n\neq A_n$,
$V=\mathbb{R}\operatorname{-span}(\ga_1,\ga_2,\ldots,
\ga_{n})$ and $\beta$ be the branch root of $X_n$.
Denote $\ga_i^\vee=2 \frac{\ga_i}{(\ga_i,\ga_i)}$ the coroots and $w_{\ga_i^\vee} \in V^*$
the corresponding weights. Write $w_{\beta^\vee}=(v,\cdot), \ v\in V$
and let $c\in \mathbb{N}$ be the smallest integer such that
$c \cdot v \in \mathbb{Z}\operatorname{-span}(\ga_1^\vee,\ga_2^\vee,\ldots, \ga_{n}^\vee)$.
Then $c\cdot v=\displaystyle{\sum_{j=1}^{n}m_j\ga_j^\vee}$
where $m_j$ are the nonzero affine exponents of $X_n^{(1)}$
and the coefficient of $\beta^\vee$ is the affine Coxeter number.

\begin{example}

Removing  the branch vertex of $B_4$ we obtain  the root
system $A_2 \times A_1$ with Coxeter polynomial $g(x)=(x^2+x+1)(x+1)$; the Coxeter
polynomial of $B_4^{(1)}$ is $(x-1)^2(x^2+x+1)(x+1)$.
The Coxeter number of $A_2 \times A_1$ is the affine Coxeter
number of $B_4^{(1)}$, that is $3 \cdot 2=6$.

The roots of the Coxeter polynomial  are $1,1,-1, \omega,\omega^2$, where $\omega$ is  a
primitive third root of unity. If $\zeta=e^{\frac{2 \pi i}{6}}$ then
$1=\zeta^{0}, \omega=\zeta^{2}, -1=\zeta^{3},\omega^2=\zeta^{4}, 1=\zeta^{6}$.  The
numbers $0,2,3,4,6$ are the affine exponents of $B^{(1)}_4\ .$

A representation of the root system $B_4$ is given by $\ga_i=e_i-e_{i+1} \in \mathbb{R}^4, \
i=1,2,3$ and $\ga_4=e_4$, with the usual inner product of $\mathbb{R}^4$. The corresponding
co-roots are $\ga_i^\vee=e_i-e_{i+1} \in \mathbb{R}^4, \ i=1,2,3$ and $\ga_4^\vee=2e_4$
(root system of type $C_4$). The branch root is the root $\ga_3$ and the corresponding co-weight
is $v=w_{\ga_3^\vee}=(1,1,1,0) \in (\mathbb{R}^4)^*$. Now $v$ does not belong to the co-root
lattice but $2v=2\ga_1^\vee+4\ga_2^\vee+6\ga_3^\vee+3\ga_4^\vee$ does.
Therefore for $B_4^{(1)}$,  $c=2$, the non
zero affine exponents are $2,3,4,6$ and the affine Coxeter number is $6$.
\end{example}

\begin{example}
For the case of $D_6$ with root system $\{\ga_1,\ga_2,\ldots,\ga_6\}$ and branch root $\ga_4$ it
can be easily verified that $v=w_{\ga_4}=(1,1,1,1,0,0)=\ga_1+2\ga_2+3\ga_3+4\ga_4+2\ga_5+2\ga_6$.
Therefore for $D_6^{(1)}, \ c=1$, the nonzero affine exponents
are $1,2,2,2,3,4$ and the affine Coxeter number is $4$.
\end{example}

\section{Associated polynomials for affine Lie Algebras}

The Coxeter polynomials for the affine Lie algebras are  well-known, see e.g. \cite{Steko}. We
display their formulas in table \ref{affinecoxeterpoly}.
The novelty of our approach is the calculation of these polynomials using Chebyshev polynomials.

\begin{table}[!htbp]
  \caption{Coxeter polynomials for Affine Graphs}   \label{affinecoxeterpoly}
  \begin{center}
  \begin{tabular}{|c|p{4cm}|p{4cm}|}
    \hline
      Dynkin Diagram &Coxeter  polynomial & Cyclotomic Factors \\ \hline
      $ A_n^{(1)} $  &
      $ (x^i-1)\cdot(x^{n+1-i}-1),\newline i=1,2,\ldots,\left\lfloor\frac{n+1}{2}\right\rfloor $ &
      $\prod_{d|i} \Phi_{d} \prod_{d|n+1-i} \Phi_{d},\newline i=1,2,\ldots,\left\lfloor\frac{n+1}
      {2}\right\rfloor$ \\
      \hline
      $ B_n^{(1)} $ &
      $(x^{n-1}-1)(x^2-1)$ &
      $ \Phi_1 \Phi_2 \prod_{d|n-1} \Phi_{d}$ \\
      \hline
      $ C_n^{(1)} $ &
      $ (x^n-1)(x-1)$  &
      $ \Phi_1 \prod_{d|n} \Phi_{d}$ \\
      \hline
      $ D_n^{(1)} $ &
      $(x^{n-2}-1)(x-1)(x+1)^2$  &
      $ \Phi_1 \Phi_2^2 \prod_{d|{n-2}} \Phi_{d} $ \\
      \hline
      $ E_6^{(1)} $ &
      $x^7+x^6-2 x^4-2 x^3+x+1$ &
      $\Phi_1^2 \Phi_2 \Phi_3^2$ \\
      \hline
      $ E_7^{(1)} $ &
      $x^8+x^7-x^5-2x^4-x^3+x+1$ &
      $\Phi_1^2 \Phi_2^2 \Phi_3\Phi_4^2$ \\
      \hline
      $ E_8^{(1)} $ &
      $x^9+x^8-x^6-x^5-x^4-x^3+x+1 $ &
      $\Phi_1^2 \Phi_2 \Phi_3\Phi_5$ \\
      \hline
      $ F_4^{(1)} $ &
      $x^5-x^3-x^2+1$ &
      $\Phi_1^2 \Phi_2 \Phi_3 $ \\
      \hline
      $ G_2^{(1)} $  &
      $x^3-x^2-x+1 $ &
      $\Phi_1^2 \Phi_2$ \\
    \hline
  \end{tabular}
  \end{center}
\end{table}

\subsection{Associated Polynomials for \texorpdfstring{$A_n^{(1)} $}{An}  }

\bigskip

In the case of $A_{n-1}^{\left(1\right)}$ we have

\noindent
$q_n\left(x\right)=2\left(T_n\left(x\right)+\left(-1\right)^{n-1}\right)$ so
$a_n\left(x\right)=2\left(T_n\left(\frac{x}{2}\right)+\left(-1\right)^{n-1}\right)$

\noindent
and $Q_n\left(x\right)=2x^n\left(T_n\left(\frac{1}{2}\left(x+\frac{1}
{x}\right)\right)+\left(-1\right)^{n-1}\right)$.

If we set $x=e^{i \theta}$ we have
\begin{gather*}
	2x^nT_n\left(\frac{1}{2}\left(x+\frac{1}{x}\right)\right)=
	2x^nT_{n}\left(\cos{\theta}\right)=2x^n\cos{\left(n\theta\right)}\\
	=2x^n\frac{1}{2}\left(e^{i n\theta}+e^{-i n\theta}\right)=x^n\left(x^n+\dfrac{1}{x^n}\right)
	=x^{2n}+1.
\end{gather*}
Therefore
\[
Q_n\left(x\right)=x^{2n}+\left(-1\right)^{n-1}2x^n+1=\left(x^n+\left(-1\right)^{n-1}\right)^2.
\]
The factorization of $Q_n$ is given by
\[
Q_n \left(x\right)=
\begin{cases}
g_2 & n \;\;   \text{even}\\
\frac{g_1}{g_2} & n\;\;\text{odd}
\end{cases},
\]
where
$
g_1=
\displaystyle{
\prod_{d|2n} \Phi_d^2,\;\;g_2=\prod_{d|n}\Phi_d^2
}.
$

In the case of $A_n^{(1)}$ (since the graph $A_n^{(1)}$ is not a tree) the Coxeter
polynomial is not uniquely defined. There are $\left\lfloor\frac{n+1}{2}\right\rfloor$ non
conjugate Coxeter elements each one producing a different Coxeter polynomial. These
polynomials are given by the formula (see \cite{Coleman})
\[
(x^i-1)\cdot(x^{n+1-i}-1),\;\;\;i=1,2,\ldots,\left\lfloor\frac{n+1}{2}\right\rfloor.
\]
The factorization of these polynomials is given by
\[
\prod_{d|i}\Phi_d(x)\prod_{d|n+1-i}\Phi_d(x),\;\;\;i=1,2,\ldots,
\left\lfloor\frac{n+1}{2}\right\rfloor
\]
and for the first values of $n$ we obtain those presented in table \ref{CoxeterA_n}.

Note that when $n$ is even the polynomial $Q_n(x)$ can be written in the form
$Q_n(x)=f_n(x^2)$ with $f_n(x)$ the Coxeter polynomial corresponding
to the largest conjugacy class of Coxeter elements.  In fact,  for $n$ even,
\bd
	f_n(x)=\left( x^{\frac{n}{2} } -1 \right)^2.
\ed

Using the formula we  have found for the polynomial $Q_n(x)$ we can calculate the roots of the
polynomial $a_n(x)$. Since  $Q_n(x)=\left(x^n+(-1)^{n-1}\right)^2$, the roots of $Q_n$
are given by
\[
\begin{split}
	e^{\frac{2k\pi i}n},&\;\;\;k=0,1,\ldots,n-1\ , \;\; \text{for n even} \\
	e^{\frac{(2k+1)\pi i}n},&\;\;\;k=0,1,\ldots,n-1\ , \;\; \text{for n odd}
\end{split}
\]
each one being a double root. Now if $r$ is a root of $a_n$, it follows that $x-r$
is a factor of $a_n$ so $x\left(x+\dfrac{1}{x}-r\right)=x^2-r x+1$
is a factor of $Q_n(x)$, meaning that $x^2-r x+1=(x-c)(x-\bar{c})$,
with $c$ being one of the roots of $Q_n$ and $r=2 \operatorname{Re}(c)$.
We conclude that the roots of $a_n$ are given by
\[
\begin{split}
	2\cos{\frac{2k\pi}n}&,\;\;\;k=0,1,\ldots, n-1\;\;\text{for $n$ even},\\
	2\cos{\frac{(2k+1)\pi}n}&,\;\;\;k=0,1,\ldots, n-1\;\;\text{for $n$ odd}.
\end{split}
\]
From the identity $\cos(-x)=\cos{x}$ it follows that the roots of $a_{2n+2}(x)$ are given by
\[
2\cos{\frac{k\pi}{n+1}},\;\;\;k=0,1,1,2,2,\ldots,n,n,n+1,
\]
where $k=0,1,1,2,2,\ldots,n,n,n+1$ are the affine exponents
and $h=n+1$ is the affine Coxeter number associated with the Coxeter polynomial $(x^{n+1}-1)^2$.

\begin{table}[!htbp]
\caption{Coxeter polynomials for $A_n^{(1)}$}
\label{CoxeterA_n}
\begin{center}
\begin{tabular}{|c|p{9.5cm}|}
    \hline
    n &  $ f_{n+1}(x)$ \\ \hline
    3 & $i=1:\; x^4-x^3-x+1=(x-1)(x^2-1)=\Phi_1^2\Phi_2$ \\
      & $i=2:\; x^4-2x^2+1=(x^2-1)(x^2-1)=\Phi_1^2\Phi_2^2$ \\ \hline
    4 & $i=1:\; x^5-x^4-x+1=(x-1)(x^4-1)=\Phi_1^2\Phi_2\Phi_4$ \\
      & $i=2:\; x^5-x^3-x^2+1=(x^2-1)(x^3-1)=\Phi_1^2\Phi_2\Phi_3$ \\ \hline
    5 & $i=1:\; x^6-x^5-x+1=(x-1)(x^5-1)=\Phi_1^2\Phi_5$ \\
      & $i=2:\; x^6-x^4-x^2+1=(x^2-1)(x^4-1)=\Phi_1^2\Phi_2^2\Phi_4$ \\
      & $i=3:\; x^6-2x^3+1=(x^3-1)(x^3-1)=\Phi_1^2\Phi_3^2$ \\ \hline
    6 & $i=1:\; x^7-x^6-x+1=(x-1)(x^6-1)=\Phi_1^2\Phi_2\Phi_3\Phi_6$ \\
      & $i=2:\; x^7-x^5-x^2+1=(x^2-1)(x^5-1)=\Phi_1^2\Phi_2\Phi_5$ \\
      & $i=3:\; x^7-x^4-x^3+1=(x^3-1)(x^4-1)=\Phi_1^2\Phi_2\Phi_3\Phi_4$ \\ \hline
    7 & $i=1:\; x^8-x^7-x+1=(x-1)(x^7-1)=\Phi_1^2\Phi_7$ \\
      & $i=2:\; x^8-x^6-x^2+1=(x^2-1)(x^6-1)=\Phi_1^2\Phi_2^2\Phi_3\Phi_6$ \\
      & $i=3:\; x^8-x^5-x^3+1=(x^3-1)(x^5-1)=\Phi_1^2\Phi_3\Phi_5$ \\
      & $i=4:\; x^8-2x^4+1=(x^4-1)(x^4-1)=\Phi_1^2\Phi_2^2\Phi_4^2$\\ \hline
    8 & $i=1:\; x^9-x^8-x+1=(x-1)(x^8-1)=\Phi_1^2\Phi_2\Phi_4\Phi_8$ \\
      & $i=2:\; x^9-x^7-x^2+1=(x^2-1)(x^7-1)=\Phi_1^2\Phi_2\Phi_7$ \\
      & $i=3:\; x^9-x^6-x^3+1=(x^3-1)(x^6-1)=\Phi_1^2\Phi_2\Phi_3^2\Phi_6$ \\
      & $i=4:\; x^9-x^5-x^4+1=(x^4-1)(x^5-1)=\Phi_1^2\Phi_2\Phi_4\Phi_5$\\ \hline
\end{tabular}
\end{center}
\end{table}

\begin{example}

In the case of  $A_5^{(1)}$  we have
\bd
	a_6(x)=\left(x^2-1\right)^2\left(x^2-4\right).
\ed
Therefore  the roots of $a_6(x)$ are
\bd
	1,1,-1,-1,2,-2,
\ed
and they have the form
\bd
	2 \cos \frac{ m_i \pi}{h},
\ed
where $m_i$ are the affine exponents  and $h$ is the affine Coxeter number associated
with the Coxeter polynomial $(x^3-1)^2$.
\end{example}

\subsection{Associated Polynomials for \texorpdfstring{$B_n^{(1)}$}{Bn}}
\bigskip

In the case of $ B_{n-1}^{(1)}$ we have
\[
q_n\left(x\right)=8x \left(x^2-1\right) U_{n-3} \left(x \right).
\]
Therefore,
\[
a_n\left(x\right)=q_n\left( \frac{x}{2}\right)=
x\left(x^2-4\right) U_{n-3}\left( \frac{x}{2}\right),
\]
and
\[
Q_n\left(x\right)=x^n \left(x^3-x-\frac{1}{x}+\frac{1}{x^3} \right)
U_{n-3} \left( \frac{1}{2} \left( x+ \frac{1}{x}\right) \right).
\]

Set $x=e^{i\theta}$ to obtain
\[
\begin{split}
	Q_n\left(x\right)&=x^{n-3} \left(x^6-x^4-x^2+1\right) U_{n-3} \left(\cos \theta\right)\\
	&= x^{n-3} \left(x^4-1\right)\left(x^2-1\right)
	\frac{ \sin\left(n-2\right) \theta}{\sin \theta}\\
	&=x^{n-3} \left(x^4-1\right)\left(x^2-1\right)\frac{\left( e^{i\left(n-2\right)\theta}-
	e^{-i\left(n-2\right)\theta}\right)}{e^{i\theta}-e^{-i \theta}}\\
	&= x^{n-3} \left(x^4-1\right)\left(x^2-1\right)\frac{x}{x^{n-2}}
	\frac{ x^{ 2\left(n-2\right)}-1 }{x^2-1}  =x^{2n}-x^{2\left(n-2\right)}-x^4+1.
\end{split}
\]
Therefore
\[
Q_n\left(x\right)=x^{2n}-x^{2\left(n-2\right)}-x^4+1=\left(x^4-1\right)
\left(x^{2\left(n-2\right)}-1\right)=\Phi_1\Phi_2\Phi_4\prod_{d|2\left(n-2\right)}\Phi_d,
\]
for all $x \in \mathbb{C}$.  The Coxeter polynomial for $B_n^{\left(1\right)}$ is then
\[
f_{n+1} \left(x\right)=x^{n+1} -x^{n-1} -x^2+1  =\left(x^{n-1}-1\right)
\left(x^2-1\right)=\displaystyle{\Phi_1\Phi_2\prod_{d|n-1}\Phi_d}
\]
and the factorization of $a_n$ is given by
\bd
	a_n(x) =\Psi_4  \prod_{\substack{ j|2(n-2) }} \Psi_j(x).
\ed

We present the factorization of $f_n(x)$ for small values of $n$.
\begin{itemize}
	\item  $B_3^{(1)}$ \ \
	$a_4=x^4-4x^2$ \ \ \ \ \
	$f_4(x)=\Phi_1^2 \Phi_2^2$
	\item  $B_4^{(1)}$ \ \
	$a_5=x^5-5x^3+4x$ \ \ \ \ \
	$f_5= \Phi_1^2 \Phi_2 \Phi_3$
	\item  $B_5^{(1)}$ \ \
	$a_6=x^6-6x^4+8x^2$ \ \ \ \ \
	$f_6=\Phi_1^2 \Phi_2^2 \Phi_4 $
	\item  $B_6^{(1)}$ \ \
	$a_7=x^7-7x^5+13 x^3-4x$ \ \ \ \ \
	$f_7=\Phi_1^2 \Phi_2 \Phi_5$
	\item  $B_7^{(1)}$ \ \
	$a_8=x^8-8x^6+19x^4-12x^2$ \  \ \ \ \
	$f_8= \Phi_1^2 \Phi_2^2\Phi_3 \Phi_6 $
	\item  $B_8^{(1)}$ \ \
	$a_9=x^9-9x^7+26x^5 -25x^3+4x$ \ \ \ \ \
	$f_9=\Phi_1^2 \Phi_2 \Phi_7$
	\item  $B_9^{(1)}$ \ \
	$a_{10}=x^{10}-10x^8+34x^6 -44x^4+16x^2 $ \ \ \ \ \
	$f_{10}=\Phi_1^2 \Phi_2^2 \Phi_4 \Phi_8 $
\end{itemize}
In general we have two cases:
\bigskip

{\bf 1}) For the case of $B_{2n+1}^{(1)}$
\bd
	a_{2n+2}(x)=x\left(x^2-4\right) U_{2n-1} \left( \frac{x}{2} \right).
\ed
Since the roots of $U_n(x)$  are
\bd
	\cos \left( \frac{k \pi}{n+1} \right),  \ \ \ \ \ k=1,2, \dots, n,
\ed
the roots of $a_{2n+2}$ are $0, \pm 2$ and
\bd
	2 \cos \frac{k \pi} {2n}, \ \ \ \ \ \ \ \ \ \ k=1,2, \dots, 2n-1.
\ed
Therefore the affine exponents are $0,1,2,\ldots,n-1,n,n,n+1,\ldots,2n-1,2n$
and the affine Coxeter number is $h=2n .$

\bigskip
{\bf 2})  For the case of $B_{2n}^{(1)}$
\bd
	a_{2n+1}(x)=x\left(x^2-4\right) U_{2n-2} \left( \frac{x}{2} \right).
\ed
Since the roots of $U_n(x)$  are
\bd
	\cos \left( \frac{k \pi}{n+1} \right),  \ \ \ \ \ k=1,2, \dots, n,
\ed
the roots of $a_{2n+1}$ are $0,\pm 2$ and
\bd
	\cos \frac{2k \pi} {2 (2n-1)}, \ \ \ \ \ \ \ \ \ \ k=1,2, \dots, 2n-2.
\ed
It follows that the affine exponents are $0,2,\ldots,2n-2,2n-1,2n,\ldots,2(2n-1)$
and the affine Coxeter number  is $h=2(2n-1).$

\subsection{Associated Polynomials for \texorpdfstring{$C_n^{(1)}$}{Cn}  }

For $ C_{n-1}^{(1)}$ we have
\[
q_n(x)= 4 \left(x^2-1\right) U_{n-2} (x).
\]
Therefore,
\[
a_n(x)=q_n\left( \frac{x}{2}\right)=\left(x^2-4\right)U_{n-2}\left(\frac{x}{2}\right),
\]
and
\[
Q_n(x)=x^n \left(x^2-2+\frac{1}{x^2} \right)  U_{n-2}
\left( \frac{1}{2} \left( x+ \frac{1}{x}\right) \right).
\]
Using  the same method as in the previous subsection we obtain
\[
Q_n(x)=x^{2n}-x^{2(n-1)}-x^2+1=\left(x^{2(n-1)}-1\right)\left(x^2-1\right)=
\Phi_1\Phi_2\prod_{d|{2(n-1)}}\Phi_d,
\]
for all $x \in \mathbb{C}$.
The Coxeter polynomial for $C_n^{(1)}$ is then
\[
f_{n+1}(x)=x^{n+1}-x^{n}-x+1=\left(x^{n}-1\right)(x-1)= \Phi_1\prod_{d|n}\Phi_d.
\]
We present the factorization of $f_n(x)$ for small values of $n$.
\begin{itemize}
	\item $C_3^{(1)}$ \ \   $a_4=x^4-5x^2+4$ \ \ \ \ \ $f_4(x)=\Phi_1^2 \Phi_3$
	\item  $C_4^{(1)}$ \ \  $a_5=x^5-6x^3+8x  $ \ \ \ \ \ $f_5= \Phi_1^2 \Phi_2 \Phi_4$
	\item  $C_5^{(1)}$ \ \  $a_6=x^6-7x^4+13x^2-4$ \ \ \ \ \  $f_6=\Phi_1^2 \Phi_5$
	\item  $C_6^{(1)}$ \ \  $a_7=x^7-8x^5+19 x^3-12x  $ \ \ \ \ \ $f_7=\Phi_1^2 \Phi_2
	\Phi_3\Phi_6$
	\item  $C_7^{(1)}$ \ \  $a_8=x^8-9x^6+26x^4-25x^2+4 $ \  \ \ \ \ $f_8= \Phi_1^2 \Phi_7$
	\item  $C_8^{(1)}$ \ \  $a_9=x^9-10x^7+34x^5 -44x^3+16x$ \ \ \ \ \ $f_9=\Phi_1^2 \Phi_2
	\Phi_4 \Phi_8$
	\item  $C_{9}^{(1)}$ \ \ $a_{10}=x^{10}-11x^8+43x^6 -70x^4+41x^2-4 $ \ \ \ \ \
	$f_{10}=\Phi_1^2 \Phi_3 \Phi_9 $
\end{itemize}
The factorization of $a_n$ is given by
\bd
	a_n(x) = \prod_{\substack{ j|2(n-1) }} \Psi_j(x).
\ed

\noindent
In general we have
\bd
	a_{n+1}(x)=\left(x^2-4\right) U_{n-1} \left( \frac{x}{2} \right)
\ed
and therefore the roots of $a_{n+1}$ are $\pm 2$ and
\bd
	2 \cos \frac{k \pi} {n} \ \ \ \ \ \ \ \ \ \ k=1,2, \dots, n-1.
\ed
The affine exponents are $0,1,\ldots,n-1,n$ and the affine Coxeter number is $h=n.$

\subsection{Associated Polynomials for \texorpdfstring{$D_n^{(1)}$}{Dn}  }

In the case of $ D_{n-1}^{(1)}$ we have
\bd
	q_n(x)= 16x^2 \left(x^2-1\right)U_{n-4} (x).
\ed
Therefore,
\bd
	a_n(x)=q_n\left( \frac{x}{2}\right)=
	x^2\left(x^2-4\right)U_{n-4}\left(\frac{x}{2}\right)
\ed
and
\bd
	Q_n(x)=x^n \left(x^4-2+\frac{1}{x^4} \right)  U_{n-4}
	\left( \frac{1}{2} \left( x+ \frac{1}{x}\right)\right).
\ed
It follows that
\[
Q_n(x)=\left(x^4-1\right)\left( x^{2(n-2)}+x^{2(n-3)}-x^2-1\right)
=\left(x^4-1\right)\left(x^2+1\right)\left(x^{2(n-3)}-1\right)
\]
and the Coxeter polynomial for $D_n^{(1)}$ is
\bd
	f_{n+1} (x)=\left(x^{n-2}-1\right)\left(x^2-1\right)\left(x+1\right).
\ed
We present the factorization of $f_n(x)$ for small values of $n$.
\begin{itemize}
	\item  $D_4^{(1)}$ \ \  $a_5=x^5-4x^3  $ \ \ \ \ \ $f_5= \Phi_1^2 \Phi_2^3$
	\item  $D_5^{(1)}$ \ \  $a_6=x^6-5x^4+4x^2$ \ \ \ \ \  $f_6=\Phi_1^2 \Phi_2^2 \Phi_3$
	\item  $D_6^{(1)}$ \ \  $a_7=x^7-6x^5+8 x^3  $ \ \ \ \ \ $f_7=\Phi_1^2 \Phi_2^3 \Phi_4$
	\item  $D_7^{(1)}$ \ \  $a_8=x^8-7x^6+13x^4-4x^2 $ \  \ \ \ \ $f_8= \Phi_1^2 \Phi_2^2
	\Phi_5$
	\item  $D_8^{(1)}$ \ \  $a_9=x^9-8x^7+19x^5 -12x^3$ \ \ \ \ \ $f_9=\Phi_1^2 \Phi_2^3
	\Phi_3 \Phi_6$
\end{itemize}
Note that the factorization of $Q_n(x)$ is
\bd
	Q_n(x)= \Phi_1 \Phi_2 \Phi_4^2 \prod_{\substack{ j|2(n-3) }} \Phi_j(x),
\ed
and the factorization of $f_n(x)$ is
\bd
	f_n(x)= \Phi_1\Phi_2^2\prod_{\substack{ j|(n-3) }} \Phi_j(x).
\ed
The corresponding factorization of $a_n(x)$ is
\bd
	a_n(x) =\Psi_4^2  \prod_{\substack{ j|2(n-3) }} \Psi_j(x).
\ed
In general we have two cases:

{\bf 1})  For the case of $D_{2n+1}^{(1)}$ we have
\bd
	a_{2n+2}(x)=x^2(x^2-4) U_{2n-2} \left( \frac{x}{2} \right).
\ed
The roots of $a_{2n+2}$ are $0,0, \pm 2$ and
\bd
	2 \cos \frac{2k \pi}{2 (2n-1)} \, , \ \ \ \ \ \ \ \ \ \ k=1,2, \dots, 2n-2.
\ed
The affine exponents are $0,2,\ldots,2n-2,2n-1,2n-1,2n,\ldots,2(2n-1)$
and the affine Coxeter number $h=2(2n-1).$

{\bf 2}) For the case of $D_{2n}^{(1)}$
\bd
	a_{2n+1}(x)=x^2(x^2-4) U_{2n-3} \left( \frac{x}{2} \right)
\ed
and the roots of $a_{2n+1}$ are $0,0, \pm 2$ and
\bd
	2 \cos \frac{k \pi} {2 (n-1)} \ \ \ \ \ \ \ \ \ \ k=1,2, \dots, 2n-3.
\ed
Therefore the affine exponents are $0,1,\ldots,n-2,n-1,n-1,n,2n-3,2n-2$
and the affine Coxeter number is $h=2n-2.$

\section{Exceptional Lie algebras}

\subsection{\texorpdfstring{$E_6^{(1)}$}{E6} Graph}

The Cartan matrix of $E_6^{(1)}$ is of type
\bd
C_{E_6^{(1)}}=
\begin{pmatrix}
	 2 & 0 & 0 & 0 & 0 & 0 &-1 \\
	 0 & 2 &-1 & 0 & 0 & 0 & 0 \\
	 0 &-1 & 2 &-1 & 0 & 0 & 0 \\
	 0 & 0 &-1 & 2 &-1 & 0 &-1 \\
	 0 & 0 & 0 &-1 & 2 &-1 & 0 \\
	 0 & 0 & 0 & 0 &-1 & 2 & 0 \\
	-1 & 0 & 0 &-1 & 0 & 0 & 2 \
\end{pmatrix}.
\ed
The polynomial $p_7(x)$ is
\bd
	p_7(x)=x^7-14x^6+78x^5-220x^4+329x^3-246x^2+72x
\ed
and the polynomial $a_7(x)$
\bd
	a_7(x)=x^7-6x^5+9x^3-4x.
\ed
The roots of $a_7(x)$ are
\bd
	0, 1, -1, 1, -1, 2, -2
\ed
i.e.
\bd
	2 \cos  \frac{ m_i \ \pi }{6},
\ed
where $m_i  \in \{0,2,2,3,4,4,6 \} $. These are the affine exponents
for $E_6^{(1)}$ and the affine
Coxeter number is $h=6$.

The Coxeter polynomial is
\bd
	f_7(x)=x^7+x^6-2 x^4-2 x^3+x+1=
	\prod_{m_i}\left(x-e^{\frac{2 m_i \pi}{6}}\right)=\Phi_1^2 \Phi_2 \Phi_3^2.
\ed

\subsection{\texorpdfstring{$E_7^{(1)}$}{E7} Graph}
The Cartan matrix of $E_7^{(1)}$ is of type
\bd
C_{E_7^{(1)}}=
\begin{pmatrix}
	 2 &-1 & 0 & 0 & 0 & 0 & 0 & 0 \\
	-1 & 2 &-1 & 0 & 0 & 0 & 0 & 0 \\
	 0 &-1 & 2 &-1 & 0 & 0 & 0 & 0 \\
	 0 & 0 &-1 & 2 &-1 & 0 & 0 &-1 \\
	 0 & 0 & 0 &-1 & 2 &-1 & 0 & 0 \\
	 0 & 0 & 0 & 0 &-1 & 2 &-1 & 0 \\
	 0 & 0 & 0 & 0 & 0 &-1 & 2 & 0 \\
	 0 & 0 & 0 &-1 & 0 & 0 & 0 & 2 \
\end{pmatrix}.
\ed
The polynomial $p_8(x)$ is
\bd
p_8(x)= x^8-16x^7+105x^6-364x^5+714x^4-784x^3+440x^2-96x
\ed
and the polynomial $a_8(x)$
\bd
a_8(x)=x^8-7x^6+14x^4-8x^2.
\ed
The roots of $a_8(x)$ are
\bd
0, 0, 1, -1,\sqrt{2},-\sqrt{2}, 2, -2
\ed
i.e.
\bd
2 \cos  \frac{ m_i \ \pi }{12},
\ed
where $m_i  \in \{0,3,4,6,6,8,9,12 \} $.
These are the affine exponents for $E_7^{(1)}$ and $h=12$ is the affine Coxeter number.

The Coxeter polynomial is
\bd
	f_8(x)=x^8+x^7-x^5-2x^4-x^3+x+1=\prod_{m_i}\left(x-e^{\frac{2 m_i \pi}{12}}\right)=
	\Phi_1^2 \Phi_2^2 \Phi_3 \Phi_4^2.
\ed

\subsection{\texorpdfstring{$E_8^{(1)}$}{E8} Graph}
The Cartan matrix for  $E_8^{(1)}$ is
\bd
\begin{pmatrix}
	 2 & 0 & 0 & 0 & 0 & 0 & 0 &-1 & 0 \\
	 0 & 2 &-1 & 0 & 0 & 0 & 0 & 0 & 0 \\
	 0 &-1 & 2 &-1 & 0 & 0 & 0 & 0 & 0 \\
	 0 & 0 &-1 & 2 &-1 & 0 & 0 & 0 &-1 \\
	 0 & 0 & 0 &-1 & 2 &-1 & 0 & 0 & 0 \\
	 0 & 0 & 0 & 0 &-1 & 2 &-1 & 0 & 0 \\
	 0 & 0 & 0 & 0 & 0 &-1 & 2 &-1 & 0 \\
	-1 & 0 & 0 & 0 & 0 & 0 &-1 & 2 & 0 \\
	 0 & 0 & 0 &-1 & 0 & 0 & 0 & 0 & 2 \
\end{pmatrix}
\ed
and
\bd
	p_9(x)=x^9-18x^8+136x^7-560x^6+1364x^5-1992x^4+1679x^3-730x^2+120x.
\ed
The polynomial $a_9(x)$ is given by
\bd
	a_9(x)=x^9-8x^7+20x^5-17x^3+4x,
\ed
with roots
\bd
	0, 1, -1, 2, -2, \frac{\sqrt{5}-1}{2}, \frac{-\sqrt{5}+1}{2},
	\frac{\sqrt{5}+1}{2},\frac{-\sqrt{5}-1}{2}
\ed
i.e.
\bd
2 \cos  \frac{ m_i \ \pi }{30}
\ed
where $m_i  \in \{0,6,10,12,15,18,20,24,30 \} $ are the affine exponents
for $E_8^{(1)}$ and $h=30$ is the affine Coxeter number.

The Coxeter polynomial is
\bd
	f_9(x)=x^9+x^8-x^6-x^5-x^4-x^3+x+1=\prod_{m_i}\left(x-e^{\frac{2 m_i \pi}{30}}\right)=
	\Phi_1^2 \Phi_2 \Phi_3 \Phi_5.
\ed

\subsection{\texorpdfstring{$F_4^{(1)}$}{F4} Graph}
The Cartan matrix for  $F_4^{(1)}$ is
\bd
\begin{pmatrix}
	 2 &-1 & 0 & 0 & 0 \\
	-1 & 2 &-2 & 0 & 0 \\
	 0 &-1 & 2 &-1 & 0 \\
	 0 & 0 &-1 & 2 &-1 \\
	 0 & 0 & 0 &-1 & 2 \\
\end{pmatrix}
\ed
with
\bd
	p_5(x)=x^5-10x^4+33x^3-38x^2+2x+12
\ed
and
\bd
	a_5(x)=x^5-5x^3+4x.
\ed
The roots of the polynomial $a_5(x)$ are
\bd
	0,1,-1,2,-2
\ed
i.e.
\bd
	2 \cos  \frac{ m_i \ \pi }{6},
\ed
where $m_i  \in \{0,2,3,4,6 \} $ are the affine exponents for $F_4^{(1)}$ and $h=6$ is the
affine Coxeter number.

The Coxeter polynomial is
\bd
	f_5(x)=x^5-x^3-x^2+1=\prod_{m_i}\left(x-e^{\frac{2 m_i \pi}{6}}\right)=
	\Phi_1^2 \Phi_2 \Phi_3.
\ed

\subsection{Cartan matrix of type \texorpdfstring{$G_2^{(1)}$}{G2}}

The Cartan matrix for  $G_2^{(1)}$ is
\bd
\begin{pmatrix}
	 2 &-1 & 0 \\
	-1 & 2 &-3 \\
	 0 &-1 & 2 \
\end{pmatrix} \
\ed
and the polynomial $p_3(x)$ is given by
\bd
	p_3(x)=x^3-6x^2+8x.
\ed
The polynomial $a_3(x)$ is
\bd
	a_3(x)=x^3-4x,
\ed
with roots
\bd
	0,2,-2
\ed
i.e.
\bd
	2 \cos  \frac{ m_i \ \pi }{2},
\ed
where $m_i  \in \{0,1,2 \} $.
These are the affine exponents for $G_2^{(1)}$ and the affine Coxeter
number is $h=2$.

The Coxeter polynomial is
\bd
	f_3(x)=x^3-x^2-x+1=\prod_{m_i}\left(x-e^{ m_i \pi}\right)=\Phi_1^2 \Phi_2.
\ed

\section{The \texorpdfstring{$A_n^{(1)}$}{An} case}

The aim of this section is to show that the formulas of Berman, Lee, Moody (see \cite{Moody})
and Steinberg (see \cite{Steinberg, Steko}), for the Coxeter polynomials can be modified
and applied to the case of $A_n^{(1)}$. In particular we show in propositions
\ref{Moody's Prop} and \ref{Steinberg's Prop} that these formulas can be used
for the explicit calculation of all the Coxeter polynomials for any affine Lie algebra.
Also we compute and list in table \ref{affinecoxetern} the affine exponents and affine Coxeter number
associated with each Coxeter polynomial of $A_n^{(1)}$.

First we fix some notation. Let $X_n^{(1)}$ be an affine Lie algebra with Cartan matrix $C$,
$\left\{\ga_1,\ga_2,\ldots,\ga_n\right\}$ a set of simple roots of the root system of type $X_n$
and $\ga_0$ minus the highest root of $X_n$.
Let $V=\mathbb R-\operatorname{span}\{\ga_0,\ga_1,\ga_2,\ldots,\ga_n\},
\ z=(z_0,z_1,\ldots,z_n)\in\mathbb Z^{n+1}$ the left zero eigenvector of
$C, \ \ga=\sum_{i=0}^nz_i\ga_i$ and $\tilde{V}=V/ \langle \ga\rangle$ as in \ref{affCoxeternumber}.
Let $\sigma=\sigma_{\pi(0)}\sigma_{\pi(1)}\sigma_{\pi(2)}\ldots\sigma_{\pi(n)}\in gl(V)
, \ \pi\in S_{n+1}$ be a Coxeter transformation of $X_n^{(1)}$. From the definition of the
simple reflections as $\sigma_j(\ga_i)=\ga_i-C_{i,j}\ga_j$ it follows that
$\sigma$ leaves $\ga$ invariant. Therefore $\sigma$ is defined on $\tilde{V}$ and
it has finite order. Its order is the affine Coxeter number $h$, associated with $\sigma$.

\subsection{Berman, Lee, Moody's method}

The defect map of the Coxeter transformation $\sigma$ is the map $\partial:V\rightarrow \mathbb{R}\ga$
defined by $\partial=Id_V-\sigma^h$. In \cite{Moody} Berman, Lee and Moody
consider all cases of affine Lie algebras where the Dynkin diagram is bipartite
and they show that for all
$i, \ \partial (\ga_i)\in c \mathbb Z\ga$ for some $c\in\mathbb N$
(for the case $A_n^{(1)}, n$ odd, they consider the defect map of the Coxeter transformation
corresponding to the largest conjugacy class).
Further, they prove that if $\beta$ is the branch root of $X$,
then $c$ is the least positive integer such that $cw_{\beta^\vee}$
belongs to the co-root lattice and $cw_{\beta^\vee}=\sum_{i=0}^nm_i\alpha_i^{\vee}$, where
$m_i$ are the affine exponents. We generalize this result to include all Coxeter transformations
corresponding to $A_n^{(1)}$, for $n$ both even and odd.

First we generalize the notion of a branch vertex of the Dynkin diagram of a simple Lie algebra
to the case of $A_n$.
\begin{definition}\label{Branch vertex general definition}
Let $\Gamma$ be a finite Dynkin diagram of type $X_n$ and $b:V(\Gamma)\rightarrow\mathbb N$ the
weight function defined in definition \ref{Branch vertex definition}. A vertex $r_i$ is said to be
a branch vertex of $\Gamma$ if $b$ attains its maximum value on $r_i$.
\end{definition}
Therefore, for the case of the Dynkin diagram of type $A_n$, all vertices are branch vertices
with $b(r_i)=2$ for all $i=1,2,\ldots,n$. Now we can extent Berman, Lee and Moody's result
to the case of $A_n^{(1)}$.

\begin{proposition}\label{Moody's Prop}
Let $\Gamma$ be the Dynkin diagram of the simple finite dimensional
Lie algebra of type $X_n$ and let $\beta=\ga_{i_0}$ be a root corresponding
to a branch vertex of $\Gamma$. If $c$ is the least positive
integer such that $cw_{\beta^\vee}$ belongs to the co-root lattice of $X_n$,
then $cw_{\beta^\vee}=\sum_{i=0}^nm_i\ga_i^\vee$, where $m_i$ are the affine exponents and
$m_{i_0}$  is the affine Coxeter number of $X_n^{(1)}$ associated with a Coxeter
polynomial of $X_n^{(1)}$.
\end{proposition}
\begin{proof}
We consider only the case $X_n=A_n$ since the other cases follow from Berman, Lee and Moody's result.

We realize $A_n$ as the set of vectors in $\mathbb R^{n+1}$ with length $\sqrt{2}$
and whose coordinates are integers and sum to zero.
The inner product is the usual inner product in $\mathbb R^{n+1}$.
One choice of  a basis of the root system of type $A_n$ is
$$
\ga_i=(\underbrace{0,0,\ldots,0,1}_{i \text{ terms}},-1,0,\ldots,0,0).
$$
The co-roots are $\ga_i^\vee=\frac{2\ga_i}{(\ga_i,\ga_i)}=\ga_i$ and the corresponding weights
are
$$
w_i=\frac{1}{n+1}(\underbrace{n+1-i,n+1-i,\ldots,n+1-i}_{i \text{ terms}},-i,-i,\ldots,-i).
$$
Let $d_{i}=\gcd(n+1-i,i)=\gcd(n+1,i)$. If we choose as branch root the root
$\beta=\ga_{i_0}$ (or $\beta=\ga_{n+1-i_0}$),
$i_0\in\left\{1,2,\ldots,\left\lfloor\frac{n+1}{2}\right\rfloor\right\}$
then the smallest positive integer
$c$ for which $cw_{\beta}$ belongs to the co-root lattice is $c=\dfrac{n+1}{d_{i_0}}$
(in the case where $n$ is odd and $i=\frac{n+1}{2}$ we have $c=2$. For the other cases $c>2$).
For that $c$ we have
\begin{gather*}
	cw_{\beta}=\left(\frac{n+1-i_0}{d_{i_0}},\frac{n+1-i_0}{d_{i_0}},\ldots,\frac{n+1-i_0}{d_{i_0}},
	-\frac{i_0}{d_{i_0}},-\frac{i_0}{d_{i_0}},\ldots,-\frac{i_0}{d_{i_0}}\right)\\
	=\frac{n+1-i_0}{d_{i_0}}\ga_1+2\frac{n+1-i_0}{d_{i_0}}\ga_2+\ldots+i_0\frac{n+1-i_0}{d_{i_0}}\ga_{i_0}+\\
	i_0\frac{n-i_0}{d_{i_0}}\ga_{i_0+1}+i_0\frac{n-1-i_0}{d_{i_0}}\ga_{i_0+2}+\ldots+\frac{i_0}{d_{i_0}}\ga_n.
\end{gather*}
The coefficients
$$
0,\frac{n+1-i_0}{d_{i_0}},2\frac{n+1-i_0}{d_{i_0}},\ldots,
i_0\frac{n+1-i_0}{d_{i_0}},i_0\frac{n-i_0}{d_{i_0}},
i_0\frac{n-1-i_0}{d_{i_0}},\ldots,\frac{i_0}{d_{i_0}}
$$
are precisely the affine exponents and $i_0\frac{n+1-i_0}{d_{i_0}}$ is the affine Coxeter number
corresponding to the Coxeter polynomial
$$
(x^{i_0}-1)(x^{n+1-i_0}-1).
$$
\end{proof}

We illustrate with two examples for the cases $A_4^{(1)}$ and $A_5^{(1)}$.
\begin{example}
In the case of the root system of type $A_4$
\vspace{2ex}

$
\put(25,20){\line(1,0){50}}
\put(85,20){\line(1,0){50}}
\put(145,20){\line(1,0){50}}
\put(20,20){\circle{10}}
\put(18,30){$1$}
\put(80,20){\circle{10}}
\put(78,30){$2$}
\put(140,20){\circle{10}}
\put(138,30){$3$}
\put(200,20){\circle{10}}
\put(198,30){$4$}
$

\noindent
the simple roots are
\begin{gather*}
	\ga_1=(1,-1,0,0,0),\ga_2=(0,1,-1,0,0),\ga_3=(0,0,1,-1,0),\ga_4=(0,0,0,1,-1).
\end{gather*}
If we choose $\ga_1$ (or $\ga_4$) as the branch root then
$$
w_1=\frac{1}{5}(4,-1,-1,-1,-1)
$$
and $5w_1=4\ga_1+3\ga_2+2\ga_3+\ga_4$. The affine Coxeter number is $4$ and the affine exponents
$0,1,2,3,4$ which give rise to the Coxeter polynomial $(x-1)(x^4-1)$.

If we choose $\ga_2$ (or $\ga_3$) as the branch root then
$$
w_2=\frac{1}{5}(3,3,-2,-2,-2)
$$
and $5w_1=3\ga_1+6\ga_2+4\ga_3+2\ga_4$. The affine Coxeter number is $6$ and the affine exponents
$0,2,3,4,6$ which give rise to the Coxeter polynomial $(x^2-1)(x^3-1)$.
\end{example}

\begin{example}
In the case of the root system of type $A_5$
\vspace{2ex}

$
\put(25,20){\line(1,0){50}}
\put(85,20){\line(1,0){50}}
\put(145,20){\line(1,0){50}}
\put(205,20){\line(1,0){50}}
\put(20,20){\circle{10}}
\put(18,30){$1$}
\put(80,20){\circle{10}}
\put(78,30){$2$}
\put(140,20){\circle{10}}
\put(138,30){$3$}
\put(200,20){\circle{10}}
\put(198,30){$4$}
\put(260,20){\circle{10}}
\put(258,30){$5$}
$

\noindent
the simple roots are
\[
\begin{split}
\ga_1=(1,-1&,0,0,0,0),\ga_2=(0,1,-1,0,0,0),\ga_3=(0,0,1,-1,0,0),\\
\ga_4&=(0,0,0,1,-1,0),\ga_5=(0,0,0,0,1,-1).
\end{split}
\]
The affine Coxeter number corresponding to the Coxeter polynomial $(x-1)(x^5-1)$ is $5$
and the affine exponents are $0,1,2,3,4,5$.
They correspond to the branch root $\ga_1$ (or $\ga_5$) for which the co-weight is
$$
w_1=\frac{1}{6}(5,-1,-1,-1,-1,-1)
$$
and $6w_1=5\ga_1+4\ga_2+3\ga_3+2\ga_4+\ga_5$.

The branch root $\ga_2$ (or $\ga_4$) corresponds to the Coxeter polynomial $(x^2-1)(x^4-1)$
which give rise to the affine Coxeter number $4$ and the affine exponents
$0,1,2,2,3,4$.

If we choose the middle root $\ga_3$, as the branch root then
$$
w_3=\frac{1}{2}(1,1,1,-1,-1,-1)
$$
and $2w_1=\ga_1+2\ga_2+3\ga_3+2\ga_4+\ga_5$.
The affine Coxeter number is $3$ and the affine exponents are
$0,1,2,3,2,1$ which give rise to the Coxeter polynomial $(x^3-1)^2$.
\end{example}

\subsection{Steinberg's method}
Steinberg in \cite{Steinberg} (see also \cite{Steko}) shows that for the affine root systems
considered in definition \ref{Branch vertex definition}, their Coxeter polynomial
is a product of Coxeter polynomials of type $A_i$. Removing the branch root,
if $g(x)$ is the Coxeter polynomial of the resulting root system then
$(x-1)^2g(x)$ is the Coxeter polynomial of $X_n^{(1)}$.

We generalize Steinberg's result to the case of root systems of type $A_n^{(1)}$.
\begin{proposition}\label{Steinberg's Prop}
Let $\beta$ be a branch root, as defined in definition \ref{Branch vertex general definition}
of a root system of type $X_n$. If $g(x)$ is the Coxeter polynomial
of the root system obtained by removing $\beta$, then $(x-1)^2g(x)$ is a
Coxeter polynomial of $X_n^{(1)}$.
\end{proposition}
\begin{proof}
For $X_n\neq A_n$ we have Steinberg's theorem. For $X_n=A_n$ if we take as branch root
$\beta=\ga_{i_0}, \ i_0\in\left\{1,2,\ldots,\left\lfloor\frac{n+1}{2}\right\rfloor\right\}$,
then the root system obtained by removing $\beta$ is $A_{i_0-1}\times A_{n-i_0}$ with Coxeter
polynomial
$$
g(x)=\left(x^{i_0-1}+x^{i_0-2}+\ldots+x+1\right)\left(x^{n-i_0}+x^{n-i_0-1}+\ldots+x+1\right).
$$
Then $(x-1)^2g(x)=\left(x^{i_0}-1\right)\left(x^{n-i_0}-1\right)$ is
one of the Coxeter polynomials of $A_n^{(1)}$
\end{proof}

\begin{example}
Case $A_4^{(1)}$

\noindent
If we choose $\ga_1$ (or $\ga_4$) as the branch root then the root system
obtained by removing $\ga_1$ is $A_3$ with Coxeter polynomial $x^3+x^2+x+1$.
We obtain the Coxeter polynomial $(x-1)^2(x^3+x^2+x+1)=(x-1)(x^4-1)$.

If we choose $\ga_2$ (or $\ga_3$) as the branch root then the root system
obtained by removing the branch root is $A_1\times A_2$ with Coxeter polynomial
$(x+1)(x^2+x+1)$. The corresponding Coxeter polynomial is $(x-1)^2(x+1)(x^2+x+1)=(x^2-1)(x^3-1)$.
\end{example}
%
%
%
%
\section{Conclusion}

We have  computed  the affine exponents and the affine Coxeter number of affine Lie algebras
using various techniques, i.e., from the spectrum of Cartan and adjacency matrices,
with the help of Chebyshev polynomials,
using the procedure of Steinberg, and also the method of Berman, Lee and Moody. We payed special attention to the non-bipartite case which is not considered elsewhere and we made use of Chebyshev polynomials. 

We can summarize with the following five characterizations.

\begin{enumerate}

\item Let $f(x)$ be the affine Coxeter polynomial
(in the case of $A_n^{(1)}$ with $n$ odd we use the Coxeter polynomial
corresponding to the largest conjugacy class). Then the roots of $f$
in terms of the exponents and Coxeter number are
\bd
	e^{\frac{2m_j \pi i }{h}}  \  .
\ed

\item Let $a_n(x)$ be the characteristic polynomial of the Coxeter adjacency matrix.
The spectrum of this polynomial is called the spectrum of the Dynkin graph.
Using the knowledge of the roots of $U_n(x)$, the Chebyshev polynomial of the second kind,
we are able to compute the roots of $a_n(x)$ and in the bipartite case
they turn out to be
$$
2 \cos{\frac{m_j \pi}{h}}    \ .
$$

\item  Let $C$ be the generalized Cartan matrix associated with the affine Lie algebra.
The  eigenvalues of $C$  in the bipartite case are
\bd
4 \cos^2{\frac{m_j\pi}{2h}}    \  .
\ed

\item Let $\Pi=\{\ga_1,\ga_2,\ldots,\ga_{n}\}$ be the simple roots of the associated
simple Lie algebra,  $V=\mathbb{R}\operatorname{-span}(\ga_1,\ga_2,\ldots,\ga_{n})$
and $\beta$ the branch root. Let $w_{\beta^\vee} \in V^*$ be the weight
corresponding to the co root $\beta^\vee$.
Then for some $c \in \mathbb N$ and a proper enumeration of $m_j$ we have
$$
c\cdot w_{\beta^\vee}=\displaystyle{\sum_{j=1}^{n}m_j\ga_j^\vee} \ ,
$$
where $c \in \mathbb{N}$ is the smallest integer such that $c\cdot w_{\beta^\vee}$
belongs to the co root lattice. The coefficient of $\beta^\vee$ is the affine
Coxeter number. This method is extended in this paper for each conjugacy class
in the Coxeter group of $A_n^{(1)}$.

\item One may use a procedure of Steinberg  which relates affine Coxeter polynomials
with the corresponding Coxeter polynomial of the reduced system
obtained by removing a branch root. Each affine Coxeter polynomial is a
product of Coxeter polynomials of type $A_n$. This method is also extended to the case of
$A_n^{(1)}$.

\end{enumerate}

\end{document}